\documentclass[fontsize=12pt,a4paper,headings=normal,
twoside=false,leqno,parskip=half-,abstract=true]{scrartcl}
\usepackage[english]{babel}
\usepackage[utf8]{inputenc}
\usepackage{hyperref}
\hypersetup{
 pdftitle={Delayed Duffing Oscillator},
 pdfauthor={Bernold Fiedler},
 colorlinks=true,
 linkcolor=blue,
 citecolor=blue,
 filecolor=blue,
 urlcolor=blue}
 
\usepackage{afterpage}
\usepackage{subcaption}

\usepackage{graphicx}
\usepackage[format=plain,labelfont=bf,font=small]{caption}
\usepackage{xcolor}
\usepackage[arrow, matrix, curve]{xy}
\usepackage{float}

\usepackage{caption}
\usepackage{enumitem}  
\usepackage{tabulary}
\usepackage{array}
\newcolumntype{N}[1]{>{\centering\arraybackslash}m{#1}}

\usepackage{amsmath,amsthm}
\swapnumbers % Switch number/label style
\usepackage{amssymb} 
\makeatletter
\newcommand{\tpitchfork}{%
  \vbox{
    \baselineskip\z@skip
    \lineskip-.52ex
    \lineskiplimit\maxdimen
    \m@th
    \ialign{##\crcr\hidewidth\smash{$-$}\hidewidth\crcr$\pitchfork$\crcr}
  }%
}
\makeatother
\usepackage{latexsym}

\usepackage[notref,notcite,color,final %make final instead of draft, by % %
]{showkeys}
\definecolor{refkey}{rgb}{1,0,0}
\definecolor{labelkey}{rgb}{1,0,0}
\usepackage{tikz}

\usepackage{ upgreek }

\usepackage[textwidth=2cm,textsize=small,backgroundcolor=none]{todonotes}

  \mathchardef\ordinarycolon\mathcode`\:
  \mathcode`\:=\string"8000
  \begingroup \catcode`\:=\active
    \gdef:{\mathrel{\mathop\ordinarycolon}}
  \endgroup

\theoremstyle{plain}
\newtheorem{thm}{Theorem}[section]

\newtheorem{prop}[thm]{Proposition}
\newtheorem{cor}[thm]{Corollary}

\makeatletter

\@addtoreset{equation}{section}
\makeatother

\begin{document}

\title{\LARGE{Coexistence of infinitely many large, stable,\\
rapidly oscillating periodic solutions\\
in time-delayed Duffing oscillators}}

\author{
 \\
Bernold Fiedler*, Alejandro López Nieto*, Richard H.~Rand**,\\
Si Mohamed Sah***, Isabelle Schneider*, Babette de Wolff* \\
{~}\\
\vspace{2cm}}

\date{version of \today}
\maketitle
\thispagestyle{empty}

\vfill

*\\
Institut für Mathematik\\
Freie Universität Berlin\\
Arnimallee 3\\ 
14195 Berlin, Germany\\
\\
**\\
Cornell University\\
Ithaca,\\
 NY 14853, USA\\
***\\
Technical University of Denmark\\
Nils Koppels Allé 404\\
2800 Kgs. Lyngby, Denmark\\
% weitere Adressen hier eintragen

%%%%%%%%%%%%%%%%%%%%%%%%%%%%%%%%%%%%%%%%%%%%%%%%%%%%%%%%%%%

\newpage
\pagestyle{plain}
\pagenumbering{roman}
\setcounter{page}{1}

\begin{abstract}
We explore stability and instability of rapidly oscillating solutions $x(t)$ for the hard spring delayed Duffing oscillator 
$$x''(t)+ ax(t)+bx(t-T)+x^3(t)=0.$$
Fix $T>0$.
We target periodic solutions $x_n(t)$ of small minimal periods $p_n=2T/n$, for integer $n\rightarrow \infty$, and with correspondingly large amplitudes.
Note how $x_n(t)$ are also marginally stable solutions, respectively, of the two standard, non-delayed, Hamiltonian Duffing oscillators
$$x''+ a x+(-1)^nbx+x^3=0.$$
Stability changes for the delayed Duffing oscillator. 
Simultaneously for all sufficiently large $n\geq n_0$, we obtain local exponential stability for $(-1)^nb<0$, and exponential instability for $(-1)^nb>0$, provided that $$0 \neq (-1)^{n+1}b\,T^2< \tfrac{3}{2}\pi^2.$$
We interpret our results in terms of noninvasive delayed feedback stabilization and destabilization for large amplitude rapidly periodic solutions of the standard Duffing oscillators. 
We conclude with numerical illustrations of our results for small and moderate $n$ which also indicate a Neimark-Sacker torus bifurcation at the validity boundary of our theoretical results.

\end{abstract}

%\newpage
\vspace{2cm}
\tableofcontents

%%%%%%%%%%%%%%%%%%%%%%%%%%%%%%%%%%%%%%%%%%%%%%%%%%%%%%%%%%%

\newpage
\pagenumbering{arabic}
\setcounter{page}{1}

\section{Introduction and main result}\label{sec1}
The \emph{Duffing oscillator} \cite{duff1918, scholarduff} is given by the special case $b=0$ of the second order pendulum equation
\begin{equation}
x''+ ax + bx(t-T)+x^3=0 \ .
\label{eq:1.1}
\end{equation}
Here we suppress time $t$ as an argument of $x=x(t)$, in absence of a delay $T>0$.
For a theoretical mechanics perspective see for example \cite{KoBr11}.
For $a\geq 0=b$, the resulting integrable Hamiltonian system with energy 
\begin{equation}
H=\tfrac{1}{2}(x')^2+ \tfrac{1}{2}ax^2+\tfrac{1}{4}x^4
\label{eq:1.2}
\end{equation}
 consists of a family $x=x(t)$ of nested periodic orbits with \emph{amplitude} $A>0$, i.e.
\begin{equation}
x(0)=A, \qquad x'(0)=0 \,.
\label{eq:1.3}
\end{equation}
The \emph{minimal period} $p=p(a,A)$ is strictly decreasing to zero, for $A\nearrow \infty$, with partial derivative $p_A<0$.
This is due to the hard spring restoring force $-ax-x^3$ of the Duffing oscillator.
For $a<0$, in contrast, the double-well potential $\tfrac{1}{2}ax^2+ \tfrac{1}{4}x^4$ of the Hamiltonian $H$ features a figure-8 pair of homoclinic lobes to the hyperbolic unstable equilibrium $x=x'=0$, each filled with nested periodic orbits of periods $p$ bounded below.
The exterior of that homoclinic pair, again, is filled with nested periodic orbits of positive Hamiltonian energy $H>0$, and with minimal periods $p\searrow 0$ for $A\nearrow \infty$.
All periodic orbits with $H>0$ are \emph{odd}:
\begin{equation}
x(t+p/2)=-x(t),
\label{eq:1.4}
\end{equation}
for all $t \in \mathbb{R}$ and any fixed real $a$.
All periodic orbits are marginally stable.
See also \cite{guho83}, \cite{scholarduff}, \cite{hish18} for further discussion of the classical Duffing oscillator.
\emph{It is our main objective, in the present paper, to study the stabilization and destabilization of the exterior periodic orbits via a nonzero \emph{delay term} $b$, in the limit of large amplitudes $A$ and for correspondingly small periods $p$.}

Pyragas control \cite{Pyr1, Pyr2, fieetal07, fieetal08} is a general device for noninvasive feedback control of periodic orbits.
In our setting, consider the standard Duffing oscillator, 
\begin{equation}
x'' + \alpha x + x^3 = u(t),
\label{eq:1.5}
\end{equation}
with a control term $u=u(t)$.
For $u\equiv 0$ we assume that the periodic orbit $x$ possesses positive energy $H>0$, so that $x$ is odd. Then \eqref{eq:1.4} implies
\begin{equation}
x(t+T)=(-1)^n x(t), \quad \textrm{for} \quad T=np/2,
\label{eq:1.6}
\end{equation}
i.e. for any integer multiple $n=1,2,3, \ldots$ of the minimal half-period $p/2$.
In particular, any delayed linear feedback control
\begin{equation}
u(t):= \kappa \cdot(x-(-1)^n x(t-T))
\label{eq:1.7}
\end{equation}
with $\kappa \in \mathbb{R}$ vanishes, i.e. becomes \emph{noninvasive}, on our target periodic orbit. 
See \cite{nakue98}, \cite{Fieetal10}, for this ``half-period" variant of Pyragas control.
See \cite{Schn13, SchnBo16} and the references there, for many more sophisticated symmetry-related refinements.
The main point is that the linear stability of the target periodic orbit $x$ of the resulting delayed Duffing oscillator
\begin{equation}
x''+(\alpha - \kappa)x + (-1)^n \kappa x (t-T)+x^3= 0
\label{eq:1.8}
\end{equation}
may well depend on the choice of the control amplitude $\kappa \in \mathbb{R}$ and of the delay $T=np/2$. And it does!
Note that \eqref{eq:1.8} takes the form \eqref{eq:1.1} with 
\begin{equation}
a= \alpha - \kappa, \quad b=(-1)^n \kappa\ .
\label{eq:1.9}
\end{equation}
Our present paper in fact establishes the simultaneous stability of an unbounded sequence of rapidly oscillating periodic solutions, alternating with an unbounded sequence of rapidly oscillating periodic solutions which are unstable.
See theorem \ref{thm:1.1} below. Here \emph{rapidly oscillating periodic solutions} are defined by minimal periods $0<p<2T$.

Our results for the second order equations \eqref{eq:1.1}, \eqref{eq:1.8} stand in marked contrast with many results in the literature on scalar delay differential equations (DDEs)
\begin{equation}
\label{eq:1del}
x'=f(x, x(t-T))
\end{equation}
for $f(0,0)=0$ and nonlinearities $f$ which are strictly monotone in the delayed variable $x(t-T)$; see \cite{Nuss74, wal83, MPNu13, wal14, mp88}.
Remarkably, all \emph{rapidly oscillating periodic solutions are linearly unstable} in this \emph{monotone feedback} setting. 
Only \emph{slowly oscillating periodic solutions}, where consecutive zeros only occur at distances $t>T$, and therefore minimal periods exceed $2T$, may be stable. 
That stability also requires negative monotone feedback, i.e.~nonlinearities $f$ which are strictly decreasing in the delayed variable $x(t-T)$.
For the construction of a nonmonotone example with $xf(x)<0$ and infinitely many stable slowly oscillating periodic solutions see \cite{vas11}. 

%Survey von Alejandro, Isabelle und Babette
In the analysis of the scalar equation \eqref{eq:1del} with monotone feedback, a crucial role is played by the \emph{zero number} $z$:
a discrete valued Lyapunov function which essentially counts the number of sign changes of solutions over time intervals of length $T$; see \cite{mp88, mpse96a}. 
This suggests that rapidly oscillating solutions of \eqref{eq:1del}, for which necessarily $z\geq2$, may well decay to a slowly oscillating periodic solution, for which $z\in\{0,1\}$.
See \cite{fiemp89} for such heteroclinic orbits. 
The opposite direction, on the other hand, is strictly forbidden by the monotone decay of the zero number.

Once the condition on monotonicity of $f$ in the delayed variable of the scalar delay equation \eqref{eq:1del} is removed, stable rapid oscillations have been observed; see \cite{IvLo99, Sto08, Sto11} and the references there. 
The constructions of such $f$ involve smooth approximations of step functions and prove the existence of at least one linearly stable rapidly oscillating periodic solution. 

Our approach preserves the monotonicity condition on $f$ but explores second order equations, instead.
For second order delay equations, the existence of several slowly and rapidly oscillating solutions has been established for the delayed van-der-Pol oscillator \cite{KiLe17}. 
Techniques involved a combination of the contraction mapping theorem and interval arithmetic. 
Stability of such solutions was not proved, but was definitely supported by numerical evidence.

For the delayed Duffing oscillator \eqref{eq:1.1} stability of large amplitude rapidly oscillating periodic solutions has been observed numerically, and supported by formal asymptotic expansions. See \cite{wacha04, HaBe12, MiChaBa15, DaShRa17}.
Similar methods have been applied by \cite{XuChu03} towards delayed feedback control of a forced van der Pol - Duffing oscillator.
Neither those (or any) numerical simulations, nor the formal methods employed so far, however, amount to a mathematical proof for the coexistence of an infinity of large stable rapidly oscillating solutions.
We close this gap of mathematical rigor in sections \ref{sec1}-\ref{sec5} of the present paper.
For further discussion of numerical aspects which illustrate and support -- but do not prove -- our results, see section \ref{sec6}.

Asymptotic stability and instability of periodic orbits $\mathbf{x}(t)$ of general retarded functional differential equations
\begin{equation}
\mathbf{\dot{x}}(t)= \mathbf{f}(\mathbf{x}_t)
\label{eq:1.10}
\end{equation}
is governed by \emph{Floquet theory}.
We recall \cite{hale77, halevl93, dieetal95}.
For $\mathcal{X}:= C^0([-T,0], \mathbb{R}^N)$ as a phase space for $\mathbf{x}_t \in \mathcal{X}, \ \mathbf{x}_t(\vartheta):= \mathbf{x}(t + \vartheta),\  -T\leq \vartheta \leq 0 $,
we obtain a nonlinear (local) solution semiflow $\mathbf{x}_t=S(t, \mathbf{x}_0)$.
Periodic orbits $\mathbf{x}(t+p)= \mathbf{x}(t)$ with minimal period $p>0$ define fixed points $\mathbf{x}_t = S(p, \mathbf{x}_t)$.
Let
\begin{equation}
\Pi= \partial_{\mathbf{x}} S(p, \mathbf{x}_0)
\label{eq:1.11}
\end{equation}
denote the linearized time-$p$ map, along our periodic orbit.
Note that $\Pi$ itself is the time-$p$ map of the linearized evolution
\begin{equation}
\mathbf{\dot{y}}(t)= f'(\mathbf{x}_t)\mathbf{y}_t
\label{eq:1.12}
\end{equation}
along the periodic orbit $\mathbf{x}_t$, where $f'(\mathbf{x}_t)$ denotes the Fréchet derivative of $f$ at $\mathbf{x}_t$.
In particular, the Arzela-Ascoli theorem implies that 
\begin{equation}
\begin{aligned}
\Pi^m:\quad &\mathcal{X}\rightarrow \mathcal{X}&\\
&\mathbf{y}_0 \mapsto \mathbf{y}_{mp}& \\
\end{aligned}
\label{eq:1.13}
\end{equation}
is compact, for $mp\geqslant T$.
Therefore, the spectrum of $\Pi^m$, and likewise of $\Pi$ itself, consists of isolated nonzero eigenvalues of finite algebraic multiplicity possibly accumulating at the spectral value $0$.
The nonzero eigenvalues $\mu$ of $\Pi$ are called \emph{Floquet multipliers} of $\mathbf{x}_t$.
Note that $\mu =1$ is a (trivial) Floquet multiplier, with eigenvector $\mathbf{y}_0(\vartheta):= \mathbf{\dot{x}}_0(\vartheta)$, for $-T\leq\vartheta\leq0$.
Moreover, $\mathbf{x}_t$ is \emph{locally asymptotically stable} if
\begin{equation}
|\mu|<1
\label{eq:1.14}
\end{equation}
holds for all Floquet multipliers, except for the trivial Floquet multiplier $\mu = 1$ which is required to be algebraically simple.
If $\mathbf{x}_t$ possesses any Floquet multiplier
\begin{equation}
|\mu|>1
\label{eq:1.15}
\end{equation}
outside the complex unit circle, then $\mathbf{x}_t$ is \emph{unstable}.
For brevity, we use \eqref{eq:1.14} and \eqref{eq:1.15} as definitions of \emph{linear stability} and \emph{linear instability}, respectively.

To formulate our main result, fix any delay $T>0$.
Let $x_n(t)$ denote the unique periodic orbit of the non-delayed Duffing oscillator
\begin{equation}
x''+(a+(-1)^nb)x+ x^3=0
\label{eq:1.16a}
\end{equation}
with strictly positive Hamiltonian
\begin{equation}
H= \tfrac{1}{2}(x')^2+ \tfrac{1}{2}(a+(-1)^nb)x^2+ \tfrac{1}{4}x^4
\label{eq:1.16b}
\end{equation}
and with minimal period
\begin{equation}
p=2T/n \ .
\label{eq:1.17}
\end{equation}
Here we have to assume that $n \in \mathbb{N}$ satisfies
\begin{equation}
n>\tfrac{1}{\pi}T \cdot (a+(-1)^nb)^{1/2},
\label{eq:1.18}
\end{equation}
in case $a+(-1)^nb>0$.
Note that \eqref{eq:1.6}, \eqref{eq:1.17} imply that the same periodic solutions $x_n(t)$ also satisfy the delayed Duffing equation \eqref{eq:1.1}.

\begin{thm} \label{thm:1.1}
In the setting \eqref{eq:1.16a}--\eqref{eq:1.18} assume $T>0$ and $b \neq 0$ satisfy
\begin{equation}
(-1)^{n+1}b\,T^2<\tfrac{3}{2}\pi^2.	
\label{eq:1.19}
\end{equation}
Moreover assume that $n\geqslant n_0$ is chosen large enough.

Then the periodic solution $x_n(t)$ of the delayed Duffing equation \eqref{eq:1.1} is
\begin{eqnarray}
\textrm{linearly stable } & \Longleftrightarrow &(-1)^nb<0, \label{eq:1.20a} \\
\textrm{linearly unstable } & \Longleftrightarrow &(-1)^nb>0,
\label{eq:1.20b}
\end{eqnarray}
in the sense of \eqref{eq:1.14}, \eqref{eq:1.15}. 
\end{thm}

Of course, condition \eqref{eq:1.19} of the theorem is trivially satisfied in case \eqref{eq:1.20b}.

The minimal periods $p=2T/n$ of our rapid periodic solutions $x_n(t)$ imply that the same $x_n(t)$ also solve the delayed Duffing equation \eqref{eq:1.1}, with the delay $T$ replaced by a new delay 
\begin{equation}
\label{eq:1.28}
\tilde{T} = T+k\,p = \nu \, T > 0\,,
\end{equation}
for any integer $k\geq -n/2$.
Here we have abbreviated $\nu := 1+2k/n$.
For fixed prescribed $T>0$, this replication in particular produces rapidly oscillating large amplitude solutions of \eqref{eq:1.1}, for a dense set of delays $\tilde{T}$ which are rational multiples of $T$. Consider any sequence $k=k_n$ and define
\begin{equation}
\label{eq:1.28a}
\underline{\nu} = \liminf \,(1+2 k_n/n)\,.
\end{equation}
For fixed odd $\nu$, for example, we may choose $k_n= n (\nu-1)/2$, to obtain constant $\tilde{T}=\nu T$, and $\underline{\nu}=\nu$.

\begin{cor} \label{cor:1.1}
In the setting \eqref{eq:1.16a}--\eqref{eq:1.18}, \eqref{eq:1.28}, \eqref{eq:1.28a}, assume $b \neq 0 < T$ and $\underline{\nu}>0$ satisfy
\begin{equation}
(-1)^{n+1}b\,\underline{\nu}^2\,T^2<\tfrac{3}{2}\pi^2.	
\label{eq:1.19cor}
\end{equation}
Moreover assume that $n\geqslant n_0$ is chosen large enough.

Then the periodic solution $x_n(t)$ of the delayed Duffing equation \eqref{eq:1.1} with delays $\tilde{T}$, instead of $T$, is
\begin{eqnarray}
\textrm{linearly stable } & \Longleftrightarrow &(-1)^nb<0, \label{eq:1.20acor} \\
\textrm{linearly unstable } & \Longleftrightarrow &(-1)^nb>0,
\label{eq:1.20bcor}
\end{eqnarray}
in the sense of \eqref{eq:1.14}, \eqref{eq:1.15}.
\end{cor}

To prove the corollary, we just note that
\begin{equation}
\label{eq:1.29}
\tilde{T} = \tilde{n}\, \tfrac{p}{2}\,, \qquad \mathrm{with} \quad \tilde{n}=n+2k\,.
\end{equation}
Therefore theorem \ref{thm:1.1} applies to $\tilde{T}, \tilde{n}$, and the corollary follows.

\begin{cor}\label{cor:1.2}
Consider the standard Duffing equation \eqref{eq:1.5} with $u\equiv 0$.
Fix $T>0$.
Let $x_n(t)$ denote the unique periodic orbit with minimal period $p=2T/n$ and positive Hamiltonian.
Let $n\geqslant n_0 > \tfrac{1}{\pi} T \sqrt{\alpha_+}$ be chosen large enough, where $\alpha_+:= \max\{0,\alpha\}$.
Assume the control amplitude $\kappa$ satisfies $-\kappa T^2<\tfrac{3}{2}\pi^2$.

Then the noninvasive delayed feedback control \eqref{eq:1.7} of \eqref{eq:1.5} makes $x_n(t)$
\begin{eqnarray}
\textrm{linearly stable } & \Longleftrightarrow & \kappa<0, \\
\label{eq:1.21}
\textrm{linearly unstable } & \Longleftrightarrow & \kappa>0.
\label{eq:1.22}
\end{eqnarray}
\end{cor}

The corollary follows from theorem \ref{thm:1.1}, for $b=(-1)^n \kappa$ as in \eqref{eq:1.9}, because neither stabilization nor destabilization depends on the coefficient $a=\alpha-\kappa$ in \eqref{eq:1.16a}, at all.

The remaining sections are organized as follows. 
In section \ref{sec2} we scale the delayed Duffing oscillator \eqref{eq:1.1} such that the amplitudes $A_n=x_n(0)$ are normalized to $1$.
In particular, this introduces a small parameter
\begin{equation}
\varepsilon:=A_n^{-2}
\label{eq:1.23}
\end{equation}
which also regularizes the linearized delay equation.
In section \ref{sec3} we use oddness of $x_n(t)$ to introduce half-period Floquet multipliers $\mu$; the Floquet multipliers for the full period $p$ described in \eqref{eq:1.11}--\eqref{eq:1.13} above become $\mu^2$ in this notation.
We then derive a characteristic equation \eqref{eq:3.13} for $\mu$, of unbounded polynomial order, which involves 
\begin{equation}
\sigma := (-\mu)^{-n}.
\label{eq:1.24}
\end{equation}
The term $\sigma$ originates from  the delay $T=np/2$ which amounts to $n$ half-periods $p/2$; see \eqref{eq:1.17}.
We treat $\sigma$ as a free complex parameter in section \ref{sec4}.
Note how $|\sigma| < 1 < |\mu| $ indicates linear instability.
This motivates $\sigma$-uniform expansions of the $2 \times 2$ Wronskian, and in particular its trace, which enters the characteristic equation.
Evaluation of the expansions, in section \ref{sec5}, provides $\varepsilon$-expansions for the critical nontrivial half-period Floquet multiplier $\mu$ near $\mu=-1$, and proves theorem \ref{thm:1.1}.
The explicit expansions for $\mu$ itself, the standard Floquet-multiplier $\mu^2$, and the Floquet multiplier $\mu^n= 1/ \sigma$ at time $T=np/2$, are summarized in theorem \ref{thm:5.1}.
We conclude with some numerical illustrations, and a comparison with earlier results, in section \ref{sec6}.

For a nontechnical summary of our results, we refer to \cite{Sahetal19}.
Our result claims to accurately analyze and predict the in-/stability of unbounded infinities of rapidly oscillating periodic solutions in the delayed Duffing equation. 
Due to the delicate technical nature of our claims, we have aimed at complete mathematical proofs, which neither rely on \emph{ad hoc} ``approximations'' nor resort to unwarranted ``simplifications''. 

\textbf{Acknowledgement.}
This work originated at the \emph{International Conference on Structural Nonlinear Dynamics and Diagnosis 2018, in memoriam Ali Nayfeh}, at Tangier, Morocco.
We are deeply indebted to Mohamed Belhaq, Abderrahim Azouani, to all organizers, and to all helpers of this outstanding conference series.
They indeed keep providing a unique platform of inspiration and highest level scientific exchange, over so many years, to the benefit of all participants.
Original typesetting was patiently accomplished by Patricia H\u{a}b\u{a}\c{s}escu. 
This work was partially supported by the Deutsche Forschungsgemeinschaft through SFB 910 project A4.
Authors RHR and SMS gratefully acknowledge support by the National Science Foundation under grant number CMMI-1634664.

\numberwithin{equation}{section}
\numberwithin{figure}{section}
\numberwithin{table}{section}

%%%%%%%%%%%%%%%%%%%%%%%%%%%%%%%%%%%%%%%%%%%%%%%%%%%%%%%

\section{Scaling}\label{sec2}
In this section we rescale the integrable Duffing oscillator
\begin{equation}
x_n''+ (a+ (-1)^nb)x_n + x^3_n=0
\label{eq:2.1}
\end{equation}
so that the periodic solution $x_n$ with minimal period $p=2T/n$, positive energy $H$, and amplitude $x_n(0)=A_n>0, \  x'_n(0)=0 $ is normalized to amplitude 1; see \eqref{eq:1.16a}--\eqref{eq:1.17}.
With the abbreviation $\alpha =a+(-1)^nb$ of \eqref{eq:1.9} and the scaling parameter $\varepsilon :=A_n^{-2}$ of \eqref{eq:1.23}, the rescaled solution
\begin{equation}
x(\varepsilon,s):=x_n(t)/A_n \, , \qquad s:=A_nt \, ,
\label{eq:2.2}
\end{equation}
obviously satisfies the rescaled Duffing equation
\begin{eqnarray}
\ddot{x}+ \varepsilon \alpha x + x^3 &=& 0 \,, \label{eq:2.3}  \\
x(\varepsilon,0)=1,\quad \dot{x}(\varepsilon, 0)& =&0 \, ,
\label{eq:2.4}
\end{eqnarray}
with $\dot{ } = \tfrac{d}{ds} \,$ and rescaled minimal period
\begin{equation}
p(\varepsilon)=2TA_n/n \, .
\label{eq:2.5}
\end{equation}

With the positive Hamiltonian energy
\begin{equation}
2H= \dot{x}^2 + \varepsilon \alpha x^2 + \tfrac{1}{2}x^4 \equiv \varepsilon \alpha + \tfrac{1}{2}
\label{eq:2.6}
\end{equation}
of \eqref{eq:2.2} on $x(\varepsilon, s)$, we immediately obtain the explicit solution
\begin{eqnarray}
x(\varepsilon, s)&=& \mathrm{cn} ((1+ \varepsilon \alpha)^{1/2}s, 1/(2(1 + \varepsilon \alpha))) \, ,
\label{eq:2.7}\\
p(\varepsilon)&=&4(1+ \varepsilon \alpha)^{-1/2} \mathrm{K}(1/(2(1+ \varepsilon \alpha)))\, , \label{eq:2.8}  \\
n&=& \varepsilon^{-1/2}\, 2T/p(\varepsilon)\ .
\label{eq:2.9}
\end{eqnarray}
Here $\mathrm{cn}(u,m)$ denotes the \emph{Jacobi elliptic function} with parameter $m$, i.e. the primitive $\int^u_0 (1-m \sin^2v)^{-1/2} dv \, .$
For $u=\pi/2$, we obtain the \emph{complete elliptic integral} $K(k^2):=\mathrm{cn}(\pi/2, k)$.
Indeed claims \eqref{eq:2.7}, \eqref{eq:2.8} follow by integration of the ODE \eqref{eq:2.6} with $x(0)=1$.
Claim \eqref{eq:2.9} follows from \eqref{eq:2.5} and the scaling $A_n=1/ \sqrt{\varepsilon}$.

Let $x_*,\  p_*$ denote the resulting limits of $x(\varepsilon, \cdot), \; p(\varepsilon)$ at $\varepsilon=0$.
Then \eqref{eq:2.7}, \eqref{eq:2.8} imply 
\begin{eqnarray}
x_*(s)&=&\mathrm{cn}(s, \tfrac{1}{2}) \label{eq:2.10} \,, \\
p_*&=& 4K ( \tfrac{1}{2}) = 7.4162987...
\label{eq:2.11}
\end{eqnarray}

In particular, the amplitude $A_n$ is of order $n$ or, in other words, $n$ is of order $1/ \sqrt{\varepsilon}$.

As an aside, we remark that the appearance of elliptic functions in \eqref{eq:2.3}, i.e. of doubly periodic meromorphic functions in complex time, is not a surprise.
Indeed, let $x(t)$ solve \eqref{eq:2.3}, this time with Dirichlet initial condition $x(0)=0$.
Then $\xi(t):=-i x(it)$ solves \eqref{eq:2.3}, with $\alpha$ replaced by $-\alpha$ and with the same initial condition for $\dot{\xi}(t)=\dot{x}(t)$.
In particular solutions which pass through zero are doubly periodic, with one real and one imaginary period, and the periods relate the two opposite signs in the Duffing equation.
See \cite{Akh} for more details, for example.

For later reference we remark that 
\begin{equation}
\int_0^{p_*/2}(\dot{x}_*)^2 ds = \tfrac{1}{6}p_* \ .
\label{eq:2.12}
\end{equation}
Indeed suppression of the index $*$, integration by parts, and $\ddot{x}=-x^3,\ 4H \equiv 1$ yield 
\begin{equation}
\int\limits_0^{p/2}  \dot{x}^2=x\dot{x} \bigg\rvert_0^{p/2}-\int\limits_0^{p/2}x\ddot{x}= \int\limits_0^{p/2} x^4 = \int\limits_0^{p/2}(4H-2\dot{x}^2)=\tfrac{p}{2}-2 \int\limits_0^{p/2} \dot{x}^2 \ ,
\label{eq:2.13}
\end{equation}
which proves claim \eqref{eq:2.12}.
%%%%%%%%%%%%%%%%%%%%%%%%%%%%%%%%%%%%%%%%%%%%%%%%%%%%%%%

\section{Floquet characteristic equation}\label{sec3}
In this section we linearize the scaled delayed Duffing oscillator 
\begin{equation}
\ddot{x}+ \varepsilon a\, x + \varepsilon b \, x (t- \varepsilon^{-1/2}T)+x^3=0
\label{eq:3.1}
\end{equation}
which results from our original Duffing setting \eqref{eq:1.1} via the normalization \eqref{eq:1.23}, \eqref{eq:2.2}. 
Here and below we write $t$ for the rescaled time $s$ again.
Note the large resulting time delay $\varepsilon^{-1/2}T$.
Linearization along the periodic orbits $x(\varepsilon, t)$, as in \eqref{eq:1.10}--\eqref{eq:1.12}, leads to the linear nonautonomous delay equation
\begin{equation}
\ddot{y}+ \varepsilon a\, y + \varepsilon b \, y (t- \varepsilon^{-1/2}T)+3x^2y=0 \, .
\label{eq:3.2}
\end{equation}
Here $x^2(t)$ possesses minimal period $p(\varepsilon)/2$, because $x(t)$ is odd.
Therefore it makes sense to study half-period Floquet multipliers, i.e. nonzero eigenvalues $\mu$ of 
\begin{equation}
\Pi := \partial_{\mathbf{x}}S(p/2, \mathbf{x}_0) \, ,
\label{eq:3.3}
\end{equation}
instead of the full period $p$ in \eqref{eq:1.11}. To deal with the large time delay
\begin{equation}
\varepsilon^{-1/2}T = A_nT=np(\varepsilon)/2
\label{eq:3.4}
\end{equation}
in \eqref{eq:3.2}, according to \eqref{eq:2.5}, we can then insert
\begin{equation}
y(t- \varepsilon^{-1/2}T)= \mu^{-n}y(t)
\label{eq:3.5}
\end{equation}
for any Floquet eigenvector $y_t$ of $\Pi$, and solve
\begin{equation}
\ddot{y}+ \varepsilon (a + (-1)^nb\, \sigma)y+ 3x^2y= 0 \, .
\label{eq:3.6}
\end{equation}
Here we have stealthily introduced
\begin{equation}
\sigma := (- \mu)^{-n} \, .
\label{eq:3.7}
\end{equation}
The crucial idea in our rigorous treatment of (half-period) Floquet multipliers $\mu$, now, is to \emph{regard} $\sigma$ \emph{as an independent complex parameter} in the second order ODE \eqref{eq:3.6}, to be recombined with $\mu$ only later.

We rewrite \eqref{eq:3.6} as a $2 \times 2$ system
\begin{equation}
\begin{aligned}
\dot{y}&= y_1 \,,  \\
\dot{y_1}&= -\varepsilon(a+(-1)^n b\, \sigma)y-3x^2y \,,
\label{eq:3.8}
\end{aligned}
\end{equation}
and denote the linear evolution of \eqref{eq:3.8} by the $2 \times 2$ Wronski matrices $W=W(\varepsilon, \sigma, t_1, t_0)$ as
\begin{equation}
z(t_1):= \bigg(  
\begin{aligned}
& y(t_1) \\
& \dot{y}(t_1)
\end{aligned} \bigg) = W(\varepsilon, \sigma, t_1, t_0)
\bigg( \begin{aligned}
& y(t_0) \\
& \dot{y}(t_0)
\end{aligned} \bigg)
\in \mathbb{C}^2 \, .
\label{eq:3.9}
\end{equation}
With the subsequent abbreviation
\begin{equation}
W(t):= W(\varepsilon, \sigma, t, 0) \, ,
\label{eq:3.10}
\end{equation}
we therefore have a half-period Floquet multiplier $\mu$ if, and only if
\begin{equation}
\mu \in \mathrm{spec} \, W(\tfrac{1}{2}p(\varepsilon)), \qquad \mathrm{with} \quad \sigma=(-\mu)^{-n} \, .
\label{eq:3.11}
\end{equation}
Since the trace of the right hand side of \eqref{eq:3.8} vanishes identically, we observe
\begin{equation}
\mathrm{det} \, W(\varepsilon, \sigma, t_1, t_0)\equiv1 \, .
\label{eq:3.12}
\end{equation}
Therefore \eqref{eq:3.11} is equivalent to the \emph{Floquet characteristic equation}
\begin{equation}
\mu^2- \textrm{tr} \, W (\varepsilon, \sigma, \tfrac{1}{2}p (\varepsilon),0) \, \mu+1= 0, \qquad \mathrm{with} \quad \sigma=(-\mu)^{-n} \, .
\label{eq:3.13}
\end{equation}

We conclude by collecting a few special properties of $W$ at $\sigma=1$ and at $\varepsilon=0$.
At parameter $\sigma=1$, we rewrite \eqref{eq:3.1}, \eqref{eq:3.6} as the original Duffing oscillator
\begin{eqnarray}
\ddot{x} + \varepsilon \alpha x + x^3 & = & 0, \qquad \alpha :=a+(-1)^nb \, , \label{eq:3.14} \\
\ddot{y} + \varepsilon \alpha y + 3x^2y & = & 0 \, .
\label{eq:3.15}
\end{eqnarray}
The two columns $z_\iota= (y, \dot{y})^T$ of $W=(z_1\, \vdots \, z_2)$ are given by their initial conditions $z_1=(1,0)^T$ and $z_2=(0,1)^T$, at $t=0$.
We reintroduce the general amplitude $A$ of (rescaled) $x=x(\varepsilon, A, t)$ via the initial condition 
\begin{equation}
x=A, \quad \dot{x}=0, \qquad \textrm{at} \ t=0 \, ,
\label{eq:3.16}
\end{equation}
and denote the partial derivative $\partial_Ax(\varepsilon, A, t)$ at (rescaled) $A=1$ by $x_A$.
Then
\begin{equation}
z_1=(x_A, \dot{x}_A)^T, \quad z_2= -\tfrac{1}{1+\varepsilon \alpha} (\dot{x}, \ddot{x})^T
\label{eq:3.17}
\end{equation}
both solve the linearization \eqref{eq:3.2} with the appropriate initial conditions, and we obtain the explicit expression
\begin{equation}
W(t)=W(\varepsilon, 1, t, 0)= 
\begin{pmatrix}
      x_A &  -\dot{x} /(1+ \varepsilon \alpha)    \\
      \dot{x}_A & -\ddot{x} /(1+ \varepsilon \alpha)  
\end{pmatrix}
\label{eq:3.18a}
\end{equation}
for the Wronskian at $\sigma=1$.

At $\varepsilon=0$, the linearization \eqref{eq:3.6} and the Floquet characteristic equation \eqref{eq:3.13} become independent of the complex ``parameter" $\sigma=(-\mu)^{-n}$, even though the delay $\varepsilon^{-1/2}T$ in \eqref{eq:3.2} becomes unbounded.
Indeed, in-/stability of $x(\varepsilon, A, t)$ at $A=1$ is only decided by $\mid \sigma \mid \leq 1$, for any $\varepsilon >0$.
This justifies the notation 
\begin{equation}
\begin{aligned}
x_*(t) & := x(0,1,t) \,, \\
W_*(t_1, t_0) & := W(0, \sigma,t_1,t_0) \,,
\label{eq:3.18b}
\end{aligned}
\end{equation}
when we address the limit $\varepsilon=0$ of original amplitudes $A_n=\varepsilon^{-1/2}\nearrow \infty$, below.
\begin{prop} \label{Prop 3.1}
Assume $\sigma =1$ in the setting \eqref{eq:3.14} -- \eqref{eq:3.18b} above, and let $p=p(\varepsilon, A)$ denote the minimal period of (rescaled) $x=x(\varepsilon, A, t)$, with partial derivative $p_A:=\partial_Ap$ at $A=1$.
Then at half-period we obtain the Wronski matrix
\begin{equation}
W(\tfrac{1}{2}p)=W(\varepsilon, 1,\tfrac{1}{2}p,0) = \bigg(
\begin{array}{cc}
-1 & \quad 0 \\
-\tfrac{1}{2}(1+ \varepsilon \alpha) p_A & \quad -1
\end{array}\bigg ) \, .
\label{eq:3.19}
\end{equation}

At $\, \varepsilon=0$, and independently of $\alpha$ and $\sigma$, the partial derivative $p_A$ of the minimal period $p_*= p(0,A)$ at $A=1$ is given by
\begin{equation}
p_A=-p_*=-4K(1/2)=-7.4162987...
\label{eq:3.20a}
\end{equation}
as detailed in \eqref{eq:2.11}.
For the Wronskians at $\varepsilon=0, \ x=x_*(t)$ we obtain
\begin{equation}
W_*(t, 0)=W(0, \sigma, t, 0) = \bigg(
\begin{array}{cc}
x+t\dot{x} & \quad - \dot{x} \\
2 \dot{x}+t \ddot{x} & \quad -\ddot{x}
\end{array}\bigg ) \, ,
\label{eq:3.20b}
\end{equation}
\begin{equation}
W_*(\tfrac{1}{2}p_*, t) =W(0, \sigma, \tfrac{1}{2}p_*, t) = \bigg(
\begin{array}{cc}
\ddot{x} & \quad - \dot{x} \\
2 \dot{x}-(\tfrac{1}{2}p_*-t)\ddot{x} & \quad -x+(\tfrac{1}{2}p_*-t)\dot{x}
\end{array}\bigg ) \, .
\label{eq:3.20c}
\end{equation}
\end{prop}
\begin{proof}
To prove claim \eqref{eq:3.19}, we just note that
\begin{equation}
x(\varepsilon, A, \tfrac{1}{2}p(\varepsilon,A)+t)=-x(\varepsilon, A, t)
\label{eq:3.21}
\end{equation}
at half-periods, by oddness of the periodic solutions $x$; see \eqref{eq:1.4}.
Differentiation of \eqref{eq:3.21} with respect to $t$, at $t=0$ and $A=1$, and subsequent insertion into \eqref{eq:3.18b} with $t=\tfrac{1}{2}p$, proves that the second column of \eqref{eq:3.19} follows from the initial condition on $z_2$.
Differentiation of \eqref{eq:3.21} with respect to $A$, at $A=1$, shows 
\begin{equation}
x_A(\varepsilon,1, \tfrac{1}{2}p+t)+ \dot{x}(\varepsilon, 1,\tfrac{1}{2}p+t)\cdot \tfrac{1}{2}p_A = -x_A(\varepsilon,1,t) \, .
\label{eq:3.22}
\end{equation}
Insertion of $t=0$, of $\dot{x}=0$ at $ \tfrac{1}{2}p$, and of $x_A=1$ at $t=0$, provides the upper left entry $x_A(\varepsilon, 1, \tfrac{1}{2}p)=-1$ of claim \eqref{eq:3.19}.
To determine the remaining lower left entry $\dot{x}_A$ we differentiate \eqref{eq:3.22} with respect to $t$, at $t=0$, and obtain
\begin{equation}
\dot{x}_A(\tfrac{1}{2}p)=-\tfrac{1}{2}p_A \cdot \ddot{x}(\tfrac{1}{2}p)- \dot{x}_A(0)= \tfrac{1}{2}p_A \cdot \ddot{x}(0) = -\tfrac{1}{2}p_A \cdot (1+ \varepsilon \alpha) \, .
\label{eq:3.23}
\end{equation}
Here we have used the initial condition  $\dot{x}_A(0)=0$, oddness \eqref{eq:3.21} to replace $\ddot{x}(\tfrac{1}{2}p) = - \ddot{x}(0)$, and the ODE \eqref{eq:3.14} with $x(0)=1$ to evaluate $\ddot{x}(0)$.
We also suppressed the arguments $A=1$ and $\varepsilon$.
This proves claim \eqref{eq:3.19}.

To prove claims \eqref{eq:3.20a} -- \eqref{eq:3.20c}, we observe that 
\begin{equation}
x^A(t):=Ax(At)
\label{eq:3.24}
\end{equation}
solves the pure cubic Duffing oscillator \eqref{eq:3.14}, for $\varepsilon = 0$, if and only if $x(t)$ does.
In particular, the scaling invariance \eqref{eq:3.24} and the initial conditions \eqref{eq:3.16} imply
\begin{equation}
x(0, A, t)= Ax_*(At) \, ,
\label{eq:3.25}
\end{equation}
where $x_*(t)=x(0,1,t)$ as in \eqref{eq:3.18b}.
In particular, the minimal periods at $\varepsilon = 0$ are given by 
\begin{equation}
p(0, A)=p_*/A \, ,
\label{eq:3.26}
\end{equation}
which proves claim \eqref{eq:3.20a}.
For the partial derivative $x_A= \partial_A x(0, A, t)$ at $A=1$ in the Wronskian \eqref{eq:3.18a} at $\varepsilon=0$, \eqref{eq:3.25} implies
\begin{equation}
x_A=x + t \dot{x}\,,
\label{eq:3.27}
\end{equation}
at $x=x_*(t)$.
Insertion into \eqref{eq:3.18a} yields \eqref{eq:3.20b}.
Since $\textrm{det} \, W_*(t,0)=1$, by \eqref{eq:3.12}, this allows us to calculate
\begin{equation}
\begin{aligned}
W_*(\tfrac{1}{2}p_*,t) =&\ W_*(\tfrac{1}{2}p_*,0) \cdot W_*(0,t)= W_*(\tfrac{1}{2}p_*,0) \cdot W_*(t,0)^{-1} = \\
=&\ \bigg(
\begin{array}{cc}
-1 & 0 \\
\tfrac{1}{2}p_* & -1
\end{array}\bigg )\cdot \bigg(
\begin{array}{cc}
-\ddot{x} & \dot{x} \\
-2\dot{x} - t \ddot{x} & x+t\dot{x}
\end{array}\bigg )  \, .
\label{eq:3.28}
\end{aligned}
\end{equation}
Here we have used \eqref{eq:3.19} and \eqref{eq:3.20a} to evaluate $W_*(\tfrac{1}{2}p_*,0)$, and \eqref{eq:3.20b} to evaluate the unimodular inverse $W_*(t,0)^{-1}$.
Performing the matrix multiplication proves the remaining claim \eqref{eq:3.20c}, and the proposition.
\end{proof}

%%%%%%%%%%%%%%%%%%%%%%%%%%%%%%%%%%%%%%%%%%%%%%%%%%%%%%%

\section{Wronski trace expansion}\label{sec4}
In the characteristic equation \eqref{eq:3.13}, the trace of the half-period Wronski matrix
\begin{equation}
W(\tfrac{1}{2}p(\varepsilon))=W(\varepsilon, \sigma, \tfrac{1}{2}p(\varepsilon),0)
\label{eq:4.1}
\end{equation}
plays the decisive role for Floquet in-/stability.
Here $W(\varepsilon, \sigma,t,0)$ accounts for the linearization \eqref{eq:3.6}, \eqref{eq:3.8} with complex parameter $|\sigma| \leqslant 1$ along the periodic solution $x=x(\varepsilon,t)$ of \eqref{eq:3.14} with minimal period $p=p(\varepsilon)$.
Note analyticity of $x,p$ and $W$ in all variables.
We collect the relevant expansions at $\sigma=1$ and at $\varepsilon =0$, as follows.
\begin{prop} \label{Prop 4.1}
Denoting partial derivatives by indices we have
\begin{equation}
W_\sigma(0, \sigma, t,0)=0 \, ,
\label{eq:4.2}
\end{equation}
at $\varepsilon=0$, for all $\sigma, t$.
For the trace of $W$ we obtain the expansion 
\begin{equation}
\mathrm{tr} \, W(\varepsilon, \sigma, \tfrac{1}{2}p(\varepsilon),0)= -2-2\varepsilon \cdot (\sigma -1) \cdot \uptau (\varepsilon, \sigma)
\label{eq:4.3}
\end{equation}
with the $\sigma$-independent limiting value 
\begin{equation}
\uptau_*:= \uptau(0, \sigma) = -\tfrac{1}{24}p^2_*\cdot(-1)^nb
\label{eq:4.4}
\end{equation}
for the analytic function $\uptau=\uptau(\varepsilon, \sigma)$.
\end{prop}
\begin{proof}
With the notation $W(t)=W(\varepsilon, \sigma, t,0)$, $ \ \dot{W}(t)= \partial_tW(\varepsilon, \sigma, t, 0)$, and $W_*$ to indicate $\varepsilon=0$, we first recall the matrix ODE
\begin{equation}
\dot{W}= \mathcal{A}W \, .
\label{eq:4.4a}
\end{equation}
Here the $2 \times 2$ matrix $\mathcal{A}= \mathcal{A}(\varepsilon, \sigma, t)$ abbreviates the right hand side of \eqref{eq:3.8}, 
\begin{equation}
\mathcal{A} =\bigg(
\begin{array}{cc}
0 & \quad 1 \\
-\varepsilon(a+(-1)^n b\, \sigma)-3x^2 & \quad 0
\end{array} \bigg)  \, ,
\label{eq:4.5}
\end{equation}
and $W(0)= \mathrm{id} \,$.
Independence of $\mathcal{A}$ on $\sigma$, at $\varepsilon=0$, proves claim \eqref{eq:4.2}; see also \eqref{eq:3.20b}.

To show claim \eqref{eq:4.3} we recall that \eqref{eq:3.19} of proposition \ref{Prop 3.1} shows 
\begin{equation}
\mathrm{tr} \, W(\varepsilon, \sigma, \tfrac{1}{2}p(\varepsilon), 0)= -2
\label{eq:4.6}
\end{equation}
at $\sigma =1$.
For $\varepsilon =0$, where $x=x_*$ and $W$ are independent of $\sigma$, insertion of $t=0$ in \eqref{eq:3.20c} shows that the trace in \eqref{eq:4.6} coincides with $\ddot{x}-x=-x^3-x=-2$ at $t=0$, likewise.
By analyticity, this proves \eqref{eq:4.3} with analytic $\uptau$.

To calculate $\uptau$ at $\varepsilon =0$, as required in \eqref{eq:4.4}, we first determine $\uptau (0, \sigma)$ at $\sigma =1$.
By \eqref{eq:4.3}, we simply have to calculate the mixed partial derivative
\begin{equation}
\uptau (0,1)= -\tfrac{1}{2} \textrm{tr} \, W_{\varepsilon \sigma}(0, 1, \tfrac{1}{2}p_*,0) \, .
\label{eq:4.7}
\end{equation}
Differentiation of \eqref{eq:4.4a} at $\varepsilon=0,\ \sigma =1$ yields the inhomogeneous linear equation
\begin{equation}
\dot{W}_{\varepsilon \sigma}= \mathcal{A}W_{\varepsilon \sigma} + \mathcal{A}_{\varepsilon \sigma}W
\label{eq:4.8}
\end{equation}
with initial condition $W_{\varepsilon \sigma}(0,1,t,0)=0$ at $t=0$.
Indeed $W_\sigma \equiv 0 \equiv \mathcal{A}_\sigma$ at $\varepsilon =0$, by \eqref{eq:4.2} and \eqref{eq:4.5}.
This prevents the terms $\mathcal{A}_{\varepsilon}W_\sigma$ and $\mathcal{A}_\sigma W_\varepsilon$ from appearing in \eqref{eq:4.8}.
Solving \eqref{eq:4.8} by variation-of-constants, suppressing indices $*$ as usual, and successively invoking \eqref{eq:3.20c}, \eqref{eq:4.5}, \eqref{eq:3.20b}, \eqref{eq:2.12}, we obtain
{
\medmuskip=0mu
\thinmuskip=0mu
\thickmuskip=0mu
\begin{equation}
\begin{aligned}
&-2 \uptau (0,1)\ =\ \textrm{tr} \int^{p/2}_0 W(\tfrac{1}{2}p,t) \ \mathcal{A}_{\varepsilon \sigma}(t)\ W(t,0)\phantom{,}dt \ = \\
&= \mathrm{tr} \int ^{p/2}_0 \bigg(
\begin{array}{cc}
\ddot{x} & \qquad -\dot{x} \\
2\dot{x}-(\tfrac{1}{2}p-t) \ddot{x} & -x+ ( \tfrac{1}{2}p-t)\dot{x}
\end{array} \bigg) \bigg(
\begin{array}{cc}
0 & 0 \\
-(-1)^nb & 0
\end{array} \bigg)
\bigg(
\begin{array}{cc}
x+t\dot{x} & -\dot{x} \\
2\dot{x}+ t\ddot{x} & -\ddot{x} 
\end{array} \bigg) dt \\
&=\ -(-1)^nb \int^{p/2}_0(-\dot{x}(x+t\dot{x})- \dot{x}(-x+ \tfrac{1}{2}(p-t)\dot{x}))\phantom{,}dt \ = \\
&=\ (-1)^nb \int^{p/2}_0 \tfrac{1}{2}p_*\dot{x}^2_*(t)\phantom{,}dt\ =\ \tfrac{1}{12}(-1)^nb p_*^2 \phantom{,} .
\label{eq:4.9}
\end{aligned}
\end{equation}
}
This proves claim \eqref{eq:4.4} at $\sigma=1$.

To show that the remaining values of $\uptau(0, \sigma)$, for $ \sigma \neq 1$, do not depend on $\sigma$, we first differentiate \eqref{eq:4.3} with respect to $\varepsilon$, at $\varepsilon =0$.
It is sufficient to prove that $W_\varepsilon (0, \sigma, \tfrac{1}{2}p_*,0)$ and $\dot{W}(0, \sigma, \tfrac{1}{2}p_*,0)$ are affine linear in $\sigma$.
The term $\dot{W}$ is actually independent of $\sigma$, by \eqref{eq:4.2}.
For $W_\varepsilon (0, \sigma, t, 0)$ we obtain, analogously to \eqref{eq:4.8}, that
\begin{equation}
\dot{W}_\varepsilon= \mathcal{A}W_\varepsilon + \mathcal{A}_\varepsilon W
\label{eq:4.10}
\end{equation}
with $W_\varepsilon = 0$ at $t=0$.
Since the lower left entry $-(a+(-1)^nb\, \sigma)- 6xx_\varepsilon$ of $\mathcal{A}_\varepsilon$ is the only nonzero entry of $\mathcal{A}_\varepsilon$, and since $\mathcal{A}$ and $W$ are independent of $\sigma$ at $\varepsilon=0$, the variations-of-constants formula shows that $W_\varepsilon(0, \sigma, \tfrac{1}{2}p_*0)$ is indeed affine linear in $\sigma$.
This  proves claim \eqref{eq:4.4}, in general, and the proposition.
\end{proof}

%%%%%%%%%%%%%%%%%%%%%%%%%%%%%%%%%%%%%%%%%%%%%%%%%%%%%%%

\section{Proof of main results}\label{sec5}
The proof of our main results is based on the Floquet characteristic equation \eqref{eq:3.13}.
Inserting the trace expansion \eqref{eq:4.3}, \eqref{eq:4.4} this equation becomes
\begin{equation}
\mu^2+2(1+(\sigma-1)\uptau \varepsilon) \mu +1 =0 \, .
\label{eq:5.0a}
\end{equation}
The algebraically double trivial solution $\mu = -1$, for $\varepsilon=0$, suggests to explore an expansion
\begin{equation}
\mu= -1 + \sqrt{\varepsilon} \eta
\label{eq:5.0b}
\end{equation}
of the half-period Floquet multiplier $\mu$.
Since $\sigma =(-\mu)^{-n}$ also depends on $\eta$, and the power $n$ itself grows like $1/ \sqrt{\varepsilon}$, the characteristic equation \eqref{eq:5.0a} is quite implicit in the scaled \emph{exponent} $\eta$.
One main tool in our analysis will be the nontrivial limits $\eta_*$ and $\sigma_*$ of $\eta$ and $\sigma$, respectively, for $\varepsilon \searrow 0$.

We prove our main result, theorem \eqref{thm:1.1}, in two steps.
Based on expansions for $\mu$ and $\sigma =(-\mu)^{-n}$, we first consider the case of small $T>0$.
In that case, theorem \ref{thm:5.1} provides a quantitative expansion of the leading un-/stable half-period Floquet multiplier $\mu$, in terms of $\sigma:=(-\mu)^{-n}$.
In proposition \ref{Prop 5.2} we then extend the resulting in-/stability to larger $T$. 
In fact, we assert that half-period Floquet multipliers $\mu$ cannot cause $\sigma$ to cross the unit circle $|\sigma|=1$, and thus cannot change the in-/stability result of theorem \ref{thm:5.1} \emph{qualitatively}, as long as the crucial condition 
\begin{equation}
0 \neq (-1)^{n+1}b\,T^2 < \tfrac{3}{2} \pi^2
\label{eq:5.1}
\end{equation}
remains valid.
\begin{thm} \label{thm:5.1}
Uniformly for bounded $|\sigma|$, for small $0<T<T_0$, and for small $0\leq \varepsilon < \varepsilon_0 := \delta_0T_0^2$, we obtain an analytic expansion $(-\mu)^{-n}= \sigma = \sigma(\sqrt{\varepsilon}, T)\in \mathbb{R}$ for the nontrivial half-period Floquet multiplier $\mu$.
For $\varepsilon=0$, the $T$-expansion reads
\begin{equation}
\sigma_*(T):= \sigma(0, T)=1- \tfrac{1}{3}(-1)^n b T^2 + \ldots 
\label{eq:5.2}
\end{equation}
Because $|\sigma|<1$ indicates instability, and  $|\sigma|>1$ indicates stability, this confirms the in-/stability claims of theorem \ref{thm:1.1} for small $T$ and $\varepsilon$; see \eqref{eq:1.20a}, \eqref{eq:1.20b}.
\end{thm}
\begin{proof}
We proceed in four steps.
First we address the quadratic Floquet characteristic equation in the form \eqref{eq:5.0a}, with analytic $\uptau= \uptau(\varepsilon, \sigma)$ from proposition \ref{Prop 4.1}.
In step 2 we insert the analytic expansion of proposition \ref{Prop 4.1} for $\sigma = (-\mu)^{-n}$, in terms of $\sqrt{\varepsilon}, T, \eta $, where $\mu = -1+ \sqrt{\varepsilon} \eta$ as in \eqref{eq:5.0b}.
This eliminates $\sigma$.
In step 3, we solve the remaining equation for $\eta= \eta(\sqrt{\varepsilon}, T)$, by the implicit function theorem.
Insertion of $\eta_*= \eta(0, T)$ and expansion with respect to $T$ will complete the proof, in step 4.

\emph{Step 1: Floquet characteristic equation}\\
Insertion of $\mu = -1 + \sqrt{\varepsilon} \eta$ into \eqref{eq:5.0a} and some cancellations yield
\begin{equation}
\eta^2-2(\sigma-1) \uptau \cdot (1- \sqrt{\varepsilon} \eta)=0 \, .
\label{eq:5.4}
\end{equation}
This uniformly quadratic equation for $\eta$, and analyticity of $\uptau = \uptau(\varepsilon, \sigma)$, guarantee that $\eta$ remains bounded, a priori, uniformly for bounded $|\sigma|$ and small $\varepsilon \geq 0$.
Solving for $\sigma -1$, we obtain the equation
\begin{equation}
\sigma-1 = \frac{\eta^2}{2 \uptau \cdot (1-\sqrt{\varepsilon} \eta)} \, ,
\label{eq:5.5}
\end{equation}
which is still implicit in $\sigma$ via the trace term $\uptau = \uptau(\varepsilon, \sigma)$.

\emph{Step 2: Expansion of $\sigma$}\\
Insertion of $\mu= -1+ \sqrt{\varepsilon} \eta$ in $\sigma = (-\mu)^{-n}$ provides
\begin{equation}
\sigma-1 = (-\mu)^{-n}-1= \exp \, (-n \, \mathrm{log} \, (1- \sqrt{\varepsilon}\eta))-1=zh(z)
\label{eq:5.6}
\end{equation}
with the abbreviations
\begin{equation}
h(z):= (\exp (z)-1)/z, \quad h(0):=1, \qquad z:=-n \, \mathrm{log}(1-\sqrt{\varepsilon} \eta)  \, .
\label{eq:5.7}
\end{equation}
Note that the auxiliary function $h(z)$ is entire.
Replacing the integer $n$ by $2T/(p(\varepsilon)\sqrt{\varepsilon})$, as in \eqref{eq:3.4}, and expanding the logarithm, we obtain
\begin{equation}
z= \tfrac{2T}{p(\varepsilon)} \eta \sum^\infty_{k=0} \tfrac{1}{k+1} (\sqrt{\varepsilon }\eta)^k \, .
\label{eq:5.8}
\end{equation}

\emph{Step 3: Elimination of $\sigma$}\\
Equating expression \eqref{eq:5.5} with \eqref{eq:5.6}, \eqref{eq:5.8}, and cancelling out the trivial multiplier case $\eta = 0$, we obtain the implicit equation
\begin{equation}
\eta = \tfrac{2T}{p(\varepsilon)} \Phi (\sqrt{\varepsilon}, T, \eta) 
\label{eq:5.9}
\end{equation}
for $\eta$. The somewhat messy analytic expression $\Phi$ is given explicitly as
\begin{equation}
\Phi=2 \uptau \cdot(1- \sqrt{\varepsilon} \eta) \cdot \bigg( \sum^\infty_{k=0} \tfrac{1}{k+1}(\sqrt{\varepsilon} \eta)^k \bigg) \cdot h \bigg( \tfrac{2T}{p(\varepsilon)} \eta \sum^\infty_{k=0} \tfrac{1}{k+1}(\sqrt{\varepsilon} \eta)^k \bigg)   \, ,
\label{eq:5.10}
\end{equation}
where $\sigma$ in $\uptau = \uptau(\varepsilon, \sigma)$ again has to be replaced by \eqref{eq:5.6}.
For our purposes, however, it is sufficient to insert $\varepsilon=0$ and note that
\begin{equation}
\begin{aligned}
\Phi(0, T, \eta) &= 2 \uptau_* h \big( \tfrac{2T}{p_*} \eta \big) \,, \\
\Phi(0, 0, \eta) &= 2 \uptau_* \, .
\end{aligned}
\label{eq:5.11}
\end{equation}
Therefore we can solve \eqref{eq:5.9} for $\eta = \eta (\sqrt{\varepsilon}, T)$, near $\varepsilon=T=0$, by the implicit function theorem.
In particular, we obtain
\begin{equation}
\eta_*(T):= \eta(0, T) = \tfrac{2T}{p_*}\cdot 2 \uptau_* + \ldots \, ,
\label{eq:5.12}
\end{equation}
to leading order in $T$.

\emph{Step 4: Expansion of $\sigma_*$}\\
Reinsertion of $\eta = \eta (\sqrt{\varepsilon}, T)$ in \eqref{eq:5.6}--\eqref{eq:5.8} provides $\sigma= \sigma(\sqrt{\varepsilon}, T)= \sigma (\sqrt{\varepsilon}, T, \eta (\sqrt{\varepsilon}, T))$.
To leading order in $T$, expansion  \eqref{eq:5.12} implies
\begin{equation}
\sigma_*(T):= \sigma(0, T, \eta_*(T)) = 1+ \big( \tfrac{2T}{p_*}\big)^2\cdot 2 \uptau_* + \ldots = 1- \tfrac{1}{3}(-1)^nb\,T^2+\ldots  
\label{eq:5.13}
\end{equation}
Here we have substituted \eqref{eq:4.4} for $\uptau_*\,$.
This proves claim \eqref{eq:5.2} and the theorem.
\end{proof}
\begin{prop} \label{Prop 5.2}
As in theorem \ref{thm:1.1}, \eqref{eq:1.19}, assume $0 \neq (-1)^{n+1}b\,T^2< \tfrac{3}{2}\pi^2$.
Then there exists a continuous function $\varepsilon_0= \varepsilon_0(T)$ such that the linear in-/stability \eqref{eq:1.20a}, \eqref{eq:1.20b} does not depend on $(\varepsilon, T)$, for $0<\varepsilon< \varepsilon_0(T)$.
Therefore in-/stability coincides with the claims of theorems \ref{thm:5.1} and \ref{thm:1.1}.
\end{prop}
\begin{proof}
Our plan of proof is the following. 
It is sufficient to show the claim for $\varepsilon = 0$.
Extension to small $\varepsilon > 0$, by the implicit function theorem applied to \eqref{eq:5.9}, then proves in-/stability as claimed in the proposition.
To address $\varepsilon=0$ we will show below that $|\sigma_*|=1$ is impossible, under assumption \eqref{eq:1.19}, except for the simple trivial half-period Floquet multiplier $\mu= -1, \ \eta=0$.
Then, the total multiplicity of solutions $\eta$ of the characteristic equation \eqref{eq:5.10} with $|\sigma_*|\leq 1$, i.e. of
\begin{equation}
\eta- \tfrac{2T}{p_*} \Phi (0, T, \eta)=0 \, ,
\label{eq:5.14}
\end{equation}
cannot change, for increasing $T>0$, as long as \eqref{eq:1.19} is not violated.
Indeed, \eqref{eq:5.14} is analytic in all variables, and therefore the total algebraic multiplicity of strictly unstable $|\sigma_*|<1$, alias $\mathrm{Re}\ \eta_*>0$ in \eqref{eq:5.15} below, remains unchanged during that homotopy of $T$.
For $\varepsilon=0$ and for small $0\leq \varepsilon < \varepsilon_0(T)$, alike, this will extend the results of theorem \ref{thm:5.1} from small $T>0$ to all $T$ satisfying assumption \eqref{eq:1.19}, as claimed in theorem \ref{thm:1.1}.

To carry out this plan, consider the homotopy of $T>0$ in the unstable case $(-1)^nb\,T^2>0$ first. 
For $\varepsilon=0$ and small $T>0$, recall that expansion \eqref{eq:5.13} revealed the only unstable Floquet multiplier $|\sigma_*|<1$ to be real, and to be given by the unique algebraically simple root $\eta= \eta_*$ of \eqref{eq:5.9}.
In particular, that root remains simple and, for $\varepsilon = 0$, extends to the full range of $T$ by our homotopy. 
We already mentioned how the implicit function theorem extends that instability to small $\varepsilon>0$.

In case $(-1)^nb\,T^2<0$, we address stability for small $\varepsilon > 0$, indirectly.
Suppose, to the contrary, that for some admissible $b,T$ there exist subsequences $n \rightarrow \infty$ with $(-1)^nb\,T^2<0$, and corresponding solutions $\eta_n$ of \eqref{eq:5.9} at $\varepsilon_n = A^{-2}_n\searrow 0$ such that $|\sigma_n|\leq 1$ for $\sigma_n:=(-\mu_n)^{-n}$.
Since $\eta_n$ remain uniformly bounded, by \eqref{eq:5.4}, we can pass to convergent subsequences $\eta_n\rightarrow \eta_* \, ,\ \sigma_n \rightarrow \sigma_*\,$.
By continuity, $\eta_*$ solves \eqref{eq:5.9} at $\varepsilon = 0$ and $|\sigma_*|\leq 1$.
For $|\sigma_*|<1$, this contradicts our homotopy result at $\varepsilon = 0$.

After these preparations it only remains to address the stability boundary $|\sigma_*|=1$, for $\varepsilon = 0$.
In that limit, \eqref{eq:5.6}--\eqref{eq:5.8} imply
\begin{equation}
\sigma_*= \exp(z_*)=\exp \, \big( \tfrac{2T}{p_*} \eta_* \big) \, .
\label{eq:5.15}
\end{equation}
Here $\eta_*$, in view of \eqref{eq:5.4}, satisfies
\begin{equation}
\eta^2_*-2(\sigma_*-1)\uptau_*=0 
\label{eq:5.16}
\end{equation}
with $\uptau_*$ from \eqref{eq:4.4}. Substitution of \eqref{eq:5.15} for $\sigma_*$ leads to the transcendental equation
\begin{equation}
\eta^2_*-2 \big( \exp \, \big( \tfrac{2T}{p_*} \eta_* \big)-1 \big) \uptau_*=0
\label{eq:5.17}
\end{equation}
with the algebraically simple trivial solution $\eta_*=0$.

We show, indirectly, that \eqref{eq:5.17} cannot possess any other purely imaginary solutions $\eta_*=i \omega \neq 0$.
Indeed any such solution would require $\mathrm{sin} \, \big(\tfrac{2T}{p_*} \omega \big)= 0 \, ,$ i.e.
\begin{equation}
\tfrac{2T}{p_*} \omega = k \pi\,, \qquad k \in \mathbb{Z} \setminus \{0\}\,,
\label{eq:5.18}
\end{equation}
to annihilate the imaginary part in \eqref{eq:5.17}.
To annihilate the real part then requires
\begin{equation}
-\omega^2-2((-1)^k-1)\uptau_* = 0 \, .
\label{eq:5.19}
\end{equation}
For even $k$ and $\omega\neq 0$, this is impossible.
Hence $k$ must be odd.
Substitution of \eqref{eq:4.4} for $\uptau_*$ implies
\begin{equation}
0<\omega^2= 4 \uptau_*=-\tfrac{1}{6}p_*^2 (-1)^nb \, .
\label{eq:5.20}
\end{equation}
Insertion of \eqref{eq:5.20} in the square of \eqref{eq:5.18} finally requires
\begin{equation}
\tfrac{2}{3}(-1)^{n+1}b T^2= k^2\pi^2 \, ,
\label{eq:5.21}
\end{equation}
for some odd integer $k$.
But this contradicts our assumption \eqref{eq:1.19} and the proposition is proved.
\end{proof}

It is worth noting how the first Hopf instability, at $k=1$ and $(-1)^{n+1} b\,T^2= \tfrac{3}{2}\pi^2$, determines the exponent $\eta_*=i\omega$ in \eqref{eq:5.18} above. 
In fact, $\varepsilon^{-1/2}T=np/2$ in \eqref{eq:3.4}, odd $n$, and $y(t-\varepsilon^{-1/2}T)=-\sigma y(t)$, in \eqref{eq:3.5}, \eqref{eq:3.7}, then suggest a minimal period $q=np$ for the pair $(x(t),y(t))$. 
For $n>4$, this indicates a torus bifurcation at a rational rotation number, with subharmonic $1:n$ resonance.

More generally, our analysis \eqref{eq:5.15}--\eqref{eq:5.21} of nonzero purely imaginary exponents $\eta$ indicates a sequence of delay-induced torus bifurcations, which increasingly destabilize large amplitude rapidly oscillating periodic solutions $x_n$ of the delayed Duffing oscillator.
The destabilizations originate from $n= \infty$, as $(-1)^{n+1}b\,T^2$ successively increases through the values $\tfrac{3}{2}(k \pi)^2$ for odd integer $k$ and large odd $n$.
See figures \ref{fig:6.5}, \ref{fig:6.6} below for illustrations of the case $k=1,\ n=33$.

\emph{Proof of theorem \ref{thm:1.1}.}
With all tools at hand, we can now summarize the proof of our main result as follows.
In section \ref{sec2}, we have rescaled the unique periodic orbits $x_n$ of the delayed Duffing equation \eqref{eq:1.1}, \eqref{eq:2.1} with large amplitude $A_n=x_n(0)>0$ and rapid minimal period $p_n=2T/n$, to become solutions of \eqref{eq:2.3}, \eqref{eq:2.4} with amplitude 1, small parameter $\varepsilon :=A^{-2}_n$, and rescaled minimal period $p(\varepsilon)= 2TA_n/n$ of order 1.
The advantage of \eqref{eq:2.3}, \eqref{eq:2.4} was that the unwieldy limit of large amplitudes $A_n$ in \eqref{eq:1.1}, \eqref{eq:2.1}, became a regular perturbation of order $\varepsilon$.
The disadvantage was the appearance of a large time delay $T/\sqrt{\varepsilon}$.
In section \ref{sec3} we have derived an expansion for the associated half-period Wronski matrix $W$ of the linearized rescaled delayed Duffing equation \eqref{eq:3.8} of \eqref{eq:2.3}, along those periodic orbits; see proposition \ref{Prop 3.1}.
The large rescaled time delay, however, caused the appearance of a term $\sigma := (- \mu)^{-n}$ in the Floquet characteristic equation \eqref{eq:3.13}. 
Up to the very end, we treated $\sigma$ as just a complex coefficient in our analysis of instability, i.e. for $| \sigma| \leq 1$.
Section \ref{sec4} provided an expansion, in terms of $\varepsilon$ and $\sigma$, of the Wronski trace $\mathrm{tr} \, W(\varepsilon, \sigma, \tfrac{1}{2}p(\varepsilon),0)$; see proposition \ref{Prop 4.1}.

At that stage it became possible to solve the full characteristic equation \eqref{eq:5.0a}, with reinserted $\sigma= (-\mu)^{-n}$, in terms of the rescaled exponent $\eta:=(\mu+1)/\sqrt{\varepsilon}$ for the nontrivial half-period Floquet multiplier $\mu$.
In fact, the implicit function theorem provided an expansion $\eta = \eta(\sqrt{\varepsilon},T)$, although limited to small $\varepsilon, T>0$.
See \eqref{eq:5.9}, \eqref{eq:5.12}.
In theorem \ref{thm:5.1}, this proved the qualitative claims of theorem \ref{thm:1.1} by a quantitative expansion \eqref{eq:5.13}, for small $\varepsilon, T>0$.

The full qualitative claims of theorem \ref{thm:1.1}, for all $0 \neq (-1)^{n+1}b\,T^2<\tfrac{3}{2} \pi^2$ as required in assumption \eqref{eq:1.19}, were only established in proposition \ref{Prop 5.2}.
In particular it followed from the homotopy to small $T$, there, that the unstable dimensions of the original periodic orbits $x_n$ with $n$ large and $(-1)^nb>0$ are all equal to 1, given by a simple real half-period Floquet multiplier $\mu < -1$.
For $(-1)^nb<0$ satisfying assumption \eqref{eq:1.19}, in contrast, stability prevailed.
This proves the main theorem \ref{thm:1.1}. \hfill $\bowtie$

%%%%%%%%%%%%%%%%%%%%%%%%%%%%%%%%%%%%%%%%%%%%%%%%%%%%%%%

\section{Numerical examples}\label{sec6}

In this section we numerically investigate the stability and instability of the rapidly oscillating  periodic solutions $x_n(t)$ of the delayed Duffing equation \eqref{eq:1.1} with parameters $a=0,\,b=1$. 
We recall that theorem \ref{thm:1.1} predicts asymptotic stability, for ``sufficiently large'' odd $n$, and instability, for even $n$. 
For ``sufficiently small'' time delays $T>0$, more specifically, theorem \ref{thm:5.1} predicts an expansion 
\begin{equation}
\label{eq:6.1}
\tilde{\eta} = \tfrac{1}{3}(-1)^n b T + \ldots = \tfrac{1}{3}(-1)^n T + \ldots
\end{equation}
of the \emph{real Floquet exponent} $\tilde{\eta}= -\tfrac{1}{T} \log |\sigma|$, which determines stability; see \eqref{eq:5.2}.
Let us illustrate those theoretical predictions.

To determine the amplitudes $A_n$ and the periodic solutions $x_n(t)$ in \eqref{eq:1.16a} with minimal period $p_n=2T/n$, we proceed as indicated in section \ref{sec2}. 
We briefly summarize these results in the original variables, prior to rescaling \eqref{eq:2.2}.

We first recall the invariant Hamiltonian \eqref{eq:1.16b} with parameters $a=0,\,b=1$ to be
\begin{equation}
H =\tfrac{1}{2}{x'}^2+ (-1)^n\,\tfrac{1}{2} x^2 + \tfrac{1}{4} x^4\,.
\label{eq:ham_1}
\end{equation}
Solving \eqref{eq:ham_1} for $x' \equiv dx/dt$, and separating variables, determines the minimal period $p_n$ of the periodic orbit $x_n(t)$ of amplitude $A_n$ to be
\begin{equation}
\frac{T}{2n} = \frac{p_n}{4} = \int^{A_n}_{0} \frac{dx}{\sqrt{\left(2H - (-1)^n\,x^2 - x^4/2  \right)}}\,.
\label{eq:per_n_q}
\end{equation}
The invariant Hamiltonian $H$ of the periodic orbit $x_n(t)$ can be evaluated at $t=0$, where $x_n = A_n$ and $x'_n = 0$, to be
\begin{equation}
H_n = (-1)^n\,\tfrac{1}{2}A_n^2 + \,\tfrac{1}{4}A_n^4\,.
\label{eq:ham_2}
\end{equation}

Replacing $H$ in \eqref{eq:per_n_q} by \eqref{eq:ham_2} yields
\begin{equation}
\frac{T}{2 n} = \int^{A_n}_{0} \frac{dx}{\sqrt{\left(A_n^2-x^2\right)\,\left((-1)^n + A_n^2/2 + x^2/2\right)}}\,.
\label{eq:per_n_q_2}
\end{equation}
Precision values of the amplitudes $A_n$ are obtained by numerical solution of the implicit integral equation \eqref{eq:per_n_q_2}, for any specific value of $n=1,2,3,\ldots$ and any time delay $T$. 
To obtain $A_n$, the elliptic integral in \eqref{eq:per_n_q_2} is numerically evaluated by the Python-based function \texttt{quad}. 
Then \texttt{fsolve} is called to determine the amplitude $A_n$ satisfying \eqref{eq:per_n_q_2}.
The routine \texttt{fsolve} is a function wrapper around MINPACK’s \texttt{hybrd} and \texttt{hybrj} algorithms. 
These algorithms, in turn, are based on Powell's hybrid method \cite{Powell}, which combines Newton's method and the steepest descent method.
As an initial guess for $A_{n}$ in \texttt{fsolve}, for any chosen values of $n$ and $T$,  we use the approximation in \cite{DaShRa17}, Eq.\,{4}, for the exact amplitude.  
To double check, we have also solved \eqref{eq:per_n_q_2} for $A_n$ by explicit inversion of the series expansion for the elliptic integral, to degree 9, using the symbolic \texttt{mathematica} package and precision evaluation of the Gamma function value $\Gamma(\tfrac{1}{4})$.
This provided the approximations given below.

To determine the exact solutions $x_n(t)$ of \eqref{eq:1.16a} with minimal periods $p_{n}=2T/n$ and amplitude $A_n$ we recall section \ref{sec2}. 
Indeed the exact solutions $x_n(t)$ are expressed by elliptic integrals similar to \eqref{eq:per_n_q_2}. 
This leads to the Jacobi elliptic cosine function $\textrm{cn}$,
\begin{equation}
x_{n}(t) = A_{n}\,\textrm{cn}(\omega_{n}\,t, m_{n}) \,.
\label{eq:sol}
\end{equation}
Here $A_{n}\,,\,\omega_{n}$\,, and $0<m_{n}<1$ are the amplitude, angular frequency, and Jacobi elliptic modulus, respectively; see \cite{Akh}. 
The three parameters are related to each other through the equations
\begin{equation}
m_{n} = \frac{A^2_{n}}{2\,((-1)^n + A^2_{n})}~~~~~~~~~\textrm{and} ~~~~~~~~~~ \omega_{n} = \sqrt{(-1)^n + A^2_{n}}\,;
\label{eq:init_2}
\end{equation}
see also \cite{Rand}, for example. Here $A_n> \sqrt{2}$ for odd $n$, to ensure $H>0$. 
In particular the single parameter $A_{n}$\,, as determined above, fully describes the exact solution \eqref{eq:sol} with prescribed minimal period $p_n=2T/n$. 
The amplitudes $A_n$, derived from \eqref{eq:per_n_q} numerically, can therefore be substituted into \eqref{eq:init_2} to obtain the reference periodic solutions $x_n(t)$ of \eqref{eq:1.1}. 

%%%%%%%%%%%%%%%%%%%%%%%%%%%
\begin{figure}[t]
	%\centering
	\includegraphics[width=\textwidth]{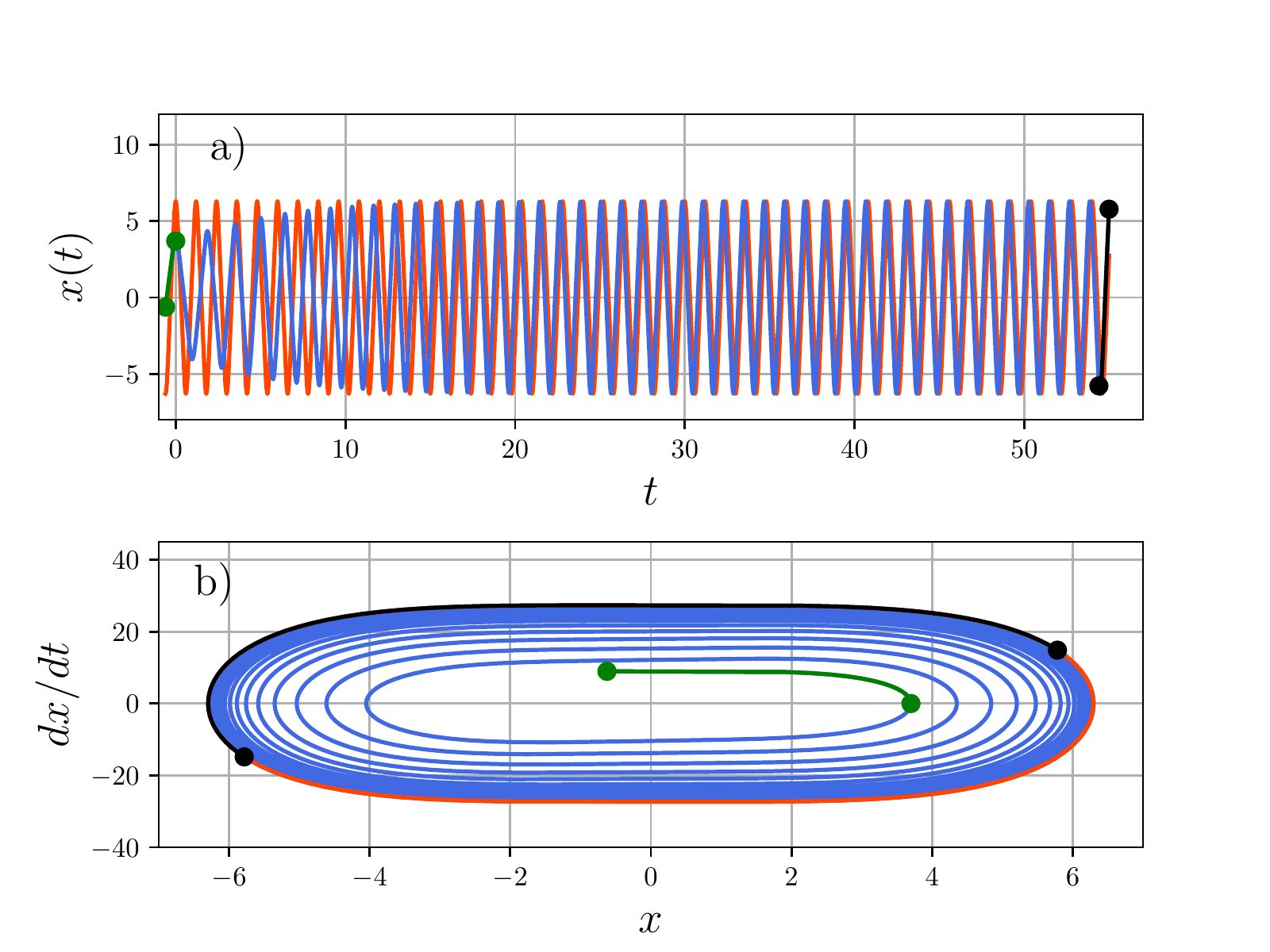} 
	\caption{
	Time histories (a) and phase plane plots (b) for $T=0.6$.
	Green: initial history function \eqref{eq:init_3} with initial amplitude $A = 3.7$. 
	Black: final state of the history function. 
	The simulated solution (blue) approaches the reference periodic solution $x_1(t)$ (red) of amplitude $A_1 = 6.29721145...$  and minimal period $p_1=1.2$ for large times $t>20$.
	}
	\label{fig:6.1}
\end{figure}
%%%%%%%%%%%%%%%%%%%%%%%%%%%%%%

Next, we numerically integrate the initial value problem for the delayed Duffing equation (1.1) using the \texttt{dde23} package. 
For the numerical integrations, we use \texttt{Pydelay} \cite{Flun11}, which is a Python library for DDEs.
The code of \texttt{dde23} is based on the Bogacki-Shampine method \cite{BoSh89} which, in turn, implements the 2(3) Runge-Kutta method.
All plots in the present work fix the maximal step size at $\Delta t = 10^{-4}$.

As initial conditions we consider the Jacobi elliptic history functions
\begin{equation}
\left(x_{0}(t), x'_{0}(t)\right) = \left( A\,\textrm{cn}(\omega t,m), -A\omega\,\textrm{sn}(\omega t,m)\,\textrm{dn}(\omega t,m)  \right),
\label{eq:init_3}
\end{equation}
of amplitude $A$, for $-T\leq t\leq 0$. 
See \eqref{eq:1.10} for the notation $x_t$\,.
Here $\omega$, $m$ are again defined via \eqref{eq:init_2}, once the initial amplitude $A$ is chosen.
In particular, initial amplitudes $A$ close to the amplitudes $A_n$ of the periodic solutions $x_n(t)$ indicate initial histories $x_0$ close to the periodic histories $(x_n)_0$ in function space $\mathcal{X} = C^0([-T,0], \mathbb{R}^2)$.

%%%%%%%%%%%%%%%%%%%%%%%%%%%
\begin{figure}[t]
	\centering
	\includegraphics[width=\textwidth]{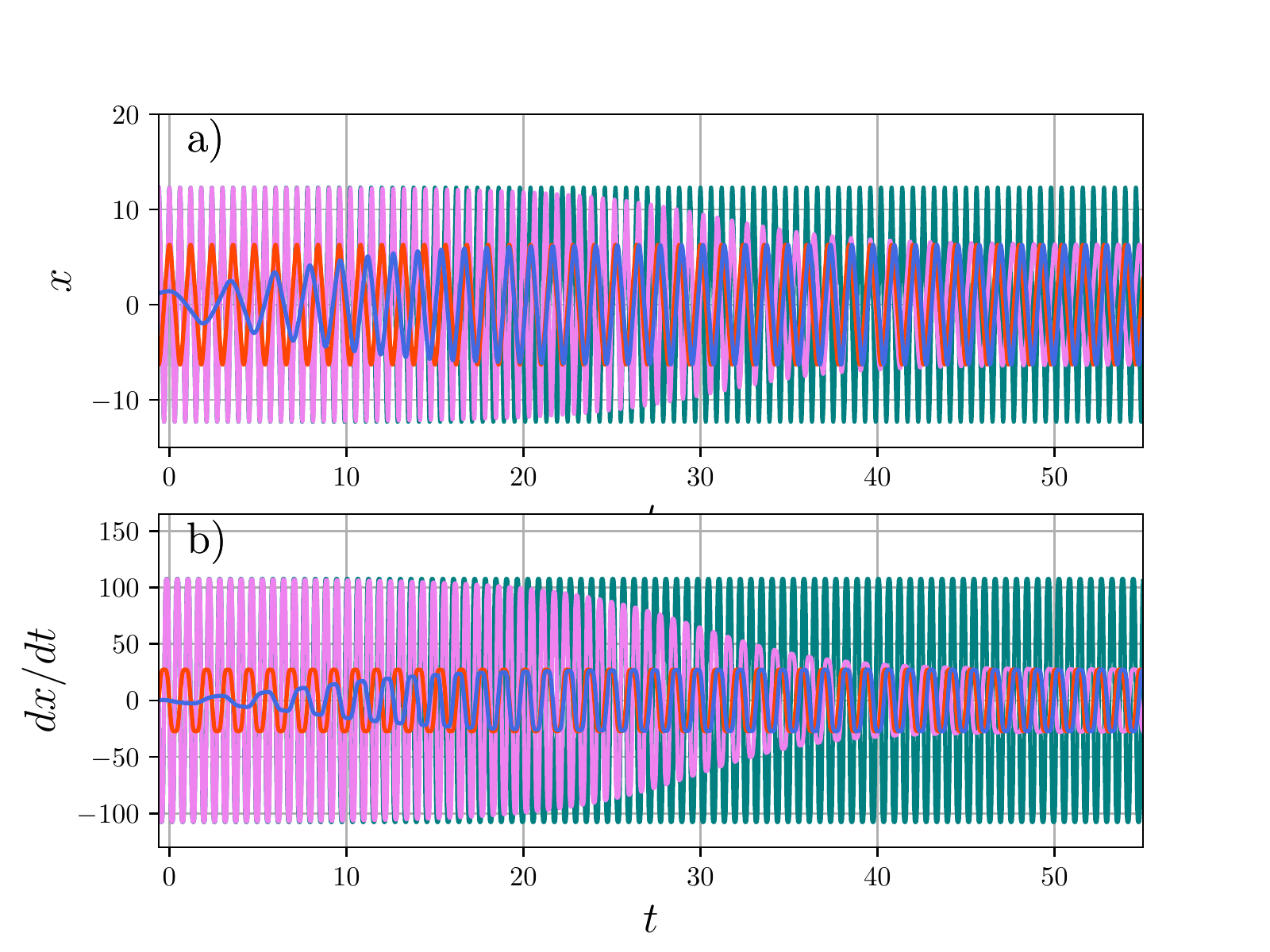} 
	\caption{
	Time histories for $T=0.6$.
	Simulations with initial history functions \eqref{eq:init_3} and initial amplitudes $A=1.42$ (blue) and $A=12.29$ (light purple).
	Time plots of $x(t)$ in (a), top, and of $\dot{x}(t)$ in (b), bottom.
	The simulated solutions indicate stability of the reference periodic solution $x_1(t)$ (red) and instability of $x_2(t)$ (teal) with amplitudes $A_1 = 6.29721145$ and $A_2 = 12.30144591494...\,$, respectively, even though $A=12.29$ is quite close to $A_2$. This suggests the light purple trajectory to be near a heteroclinic orbit from $x_2(t)$ to $x_1(t)$.}
	\label{fig:6.2}
\end{figure}
%%%%%%%%%%%%%%%%%%%%%%%%%%%%%%

For delay $T= 0.6$, figure  \ref{fig:6.1} compares a numerical solution  (blue) of the delayed Duffing equation \eqref{eq:1.1} with the reference periodic solution $x_n(t)$ (red) for $n=1$.  The figure shows time history (a) and phase plane (b).  
The green curve denotes the initial history function \eqref{eq:init_3} with initial amplitude $A = 3.7$. 
The black curve denotes the final state of the history function. 
The simulated solution (blue) approaches the reference periodic solution (red) for large times $t>20$.
Locally, but neither for small $n=1$ nor for the large initial deviations $|A-A_n|$ tested here, this is predicted by asymptotic stability of the periodic solution $x_n(t)$ for odd $n$, according to theorem \ref{thm:1.1}.

%%%%%%%%%%%%%%%%%%%%%%%%%%%%%%

\begin{figure}[t] 
\centering
	\includegraphics[width=\textwidth]{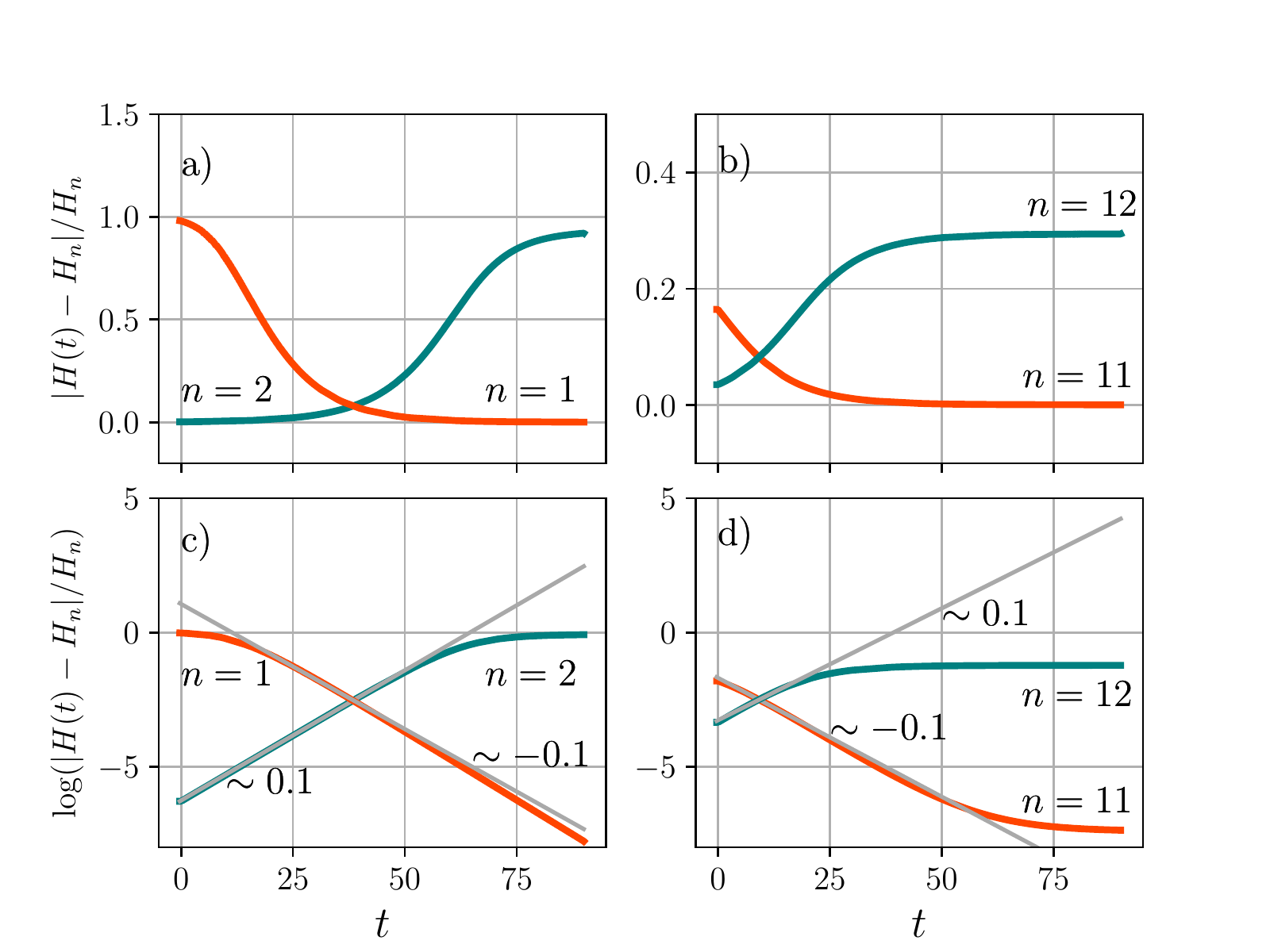}
	\caption{Evolution of the Hamiltonian $H(t)$ for $T=0.3$.
	Top row: relative deviations of $H(t)$ from the limit $H_n$, for $t\rightarrow \pm\infty$.
	Bottom row: logarithmic relative deviations.
	Orange curves: convergence to $H_n$, for $t\rightarrow +\infty$ and odd $n$.
	Teal curves: instability of $H_n$, for even $n$.
	Asymptotic slopes $\sim\pm 0.1$ (gray) confirm the Floquet exponents predicted by Theorem \ref{thm:5.1}.}
	
	\label{fig:6.3} 
\end{figure}

%%%%%%%%%%%%%%%%%%%%%%%%%%%%%%

Again for delay $T=0.6$, figure \ref{fig:6.2} shows the time histories of two numerical solutions of the delayed Duffing equation \eqref{eq:1.1}  with initial history functions \eqref{eq:init_3} and initial amplitudes $A=1.42$ (blue) and $A=12.29$ (light purple), respectively.
The  amplitudes $A_n$ of the reference periodic solutions $x_n(t)$ for $n=1$ (red) and $n=2$ (teal), respectively, are $A_1 = 6.29721145...$ and $A_2 = 12.30144591494...\,$.
Figure \ref{fig:6.2} shows how the simulated (blue, light purple) solutions for both initial amplitudes approach the same reference (red) periodic solution $x_1(t)$. 
Also note how the simulated solution with initial amplitude $A=12.29$, quite close to the periodic amplitude $A_2 = 12.30144591494...$, actually diverges from the reference (teal) periodic solution $x_2(t)$.
Again, this confirms the asymptotic stability of the periodic solution $x_n(t)$ for $n=1$, and instability for $n=2$, as predicted by theorem \ref{thm:1.1}.
The global feature of heteroclinicity from $x_2$ to $x_1$, manifested by the light purple orbit, is beyond our present scope, of course.

To test the expansion \eqref{eq:6.1} of theorem \ref{thm:5.1} for the Floquet exponent $\tilde{\eta}$ of periodic solutions $x_n(t)$, we track the Hamiltonian \eqref{eq:ham_1} numerically. 
See figures \ref{fig:6.3} and \ref{fig:6.4} for illustrations, as detailed below.
Tracking the Hamiltonian eliminates the lack of convergence in phase, which is due to the trivial Floquet exponent $\tilde{\eta}=0$.
Indeed, let $H_n>0$ denote the time-independent Hamiltonian on $x_n(t)$; see \eqref{eq:ham_2}.
Then $|H(t)-H_n(t)|$ indicates the distance of our numerical solution $x(t)$ for the delayed Duffing equation \eqref{eq:1.1} from the reference periodic \emph{orbit} $\{x_n(t)\,|\, t\in \mathbb{R}\}$, as a set, rather than the distance from any particular point $x_n(t)$ on that orbit.

We track the time-dependent Hamiltonian $H=H(t)$ as it exponentially converges to (orange curves), or diverges from (teal curves), the stationary limit $H_n$. 
According to theorem \ref{thm:5.1}, this occurs for odd $n$ (orange) and even $n$ (teal), respectively.
Let us be a little more specific.
For odd $n$ (orange), we start with initial amplitudes $A$ slightly below $A_n$ and observe convergence to $H_n$.
For even $n$ (teal), we start with initial amplitudes $A$ slightly below $A_n$ and observe convergence to $H_{n-1}$.

%%%%%%%%%%%%%%%%%%%%%%%%%%%%%%

\begin{figure}[t]
\centering 
	\includegraphics[width=\textwidth]{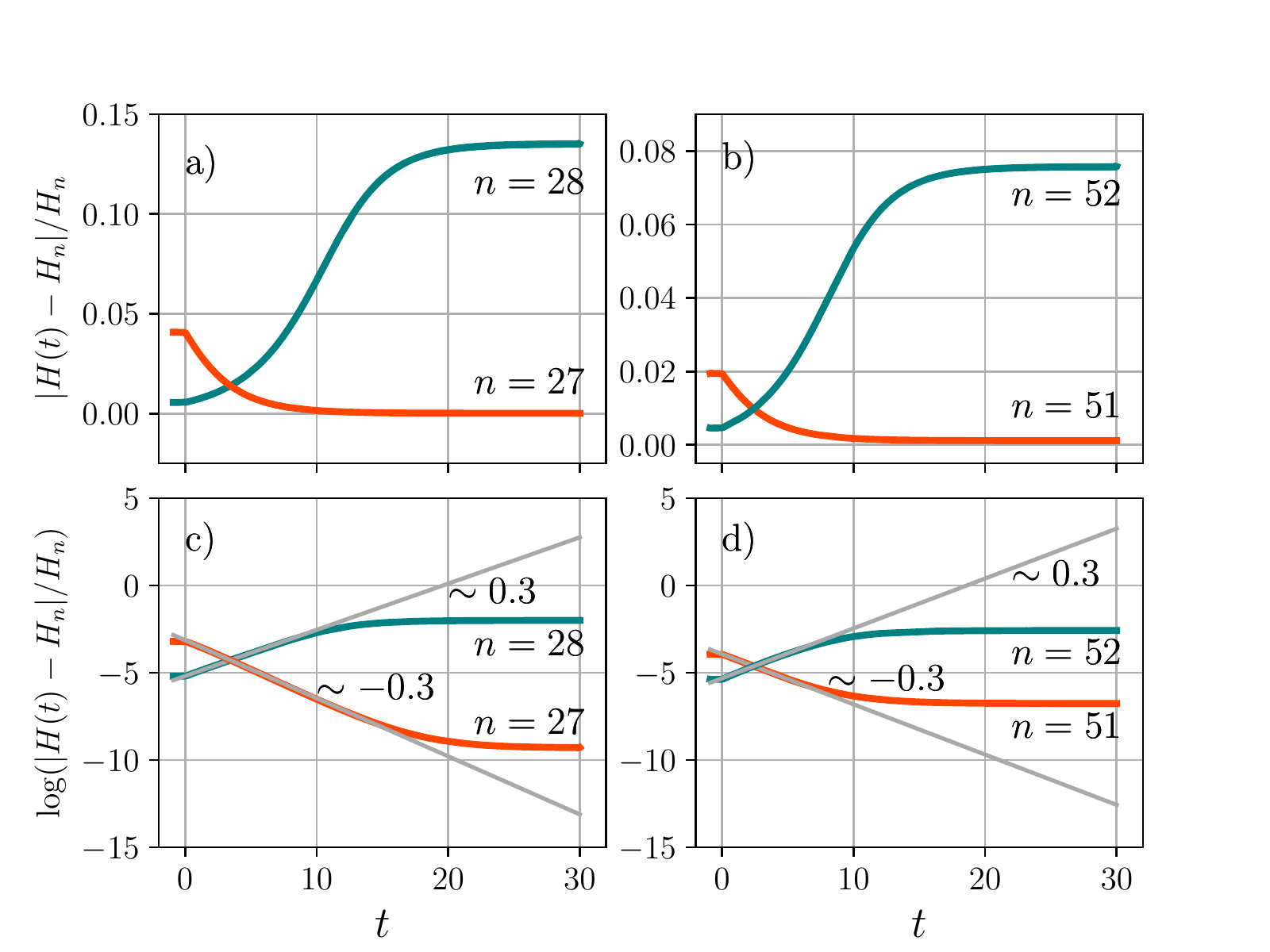}
	\caption{Evolution of the Hamiltonian $H(t)$ for $T=0.9$. Legend as in figure \ref{fig:6.3}.}
	\label{fig:6.4} 
\end{figure}

%%%%%%%%%%%%%%%%%%%%%%%%%%%%%%

The slope of $\log |H(t)-H_n|$, asymptotically with respect to time $t \rightarrow \pm \infty$, then coincides with the Floquet exponent $\tilde{\eta}$.
This determines the instability or stability of the periodic solution $x_n(t)$, depending on the positive or negative sign of the slope. 

Figures \ref{fig:6.3} and \ref{fig:6.4} show  examples of the simulated Hamiltonian $H(t)$ and of the constant reference Hamiltonian $H_{n}$\,. 
The figures confirm that periodic solutions $x_n(t)$ are stable for odd $n$ (orange curves), while even $n$ are unstable (teal curves).

In figure \ref{fig:6.3}, with delay $T=0.3$, precision amplitudes $A_n$ of the stable orange reference periodic solutions are 
$A_1=12.41931822569...$ and $A_{11}=135.97083402978303460...$; 
the amplitudes of the unstable teal periodic solutions are 
$A_2=24.69151341060282...$ and $A_{12}=148.32106281755626611...\ $.
Initial conditions are $A=4.68$ and $A=130$  for the orange curves converging to $H_1$ and $H_{11}$\,, respectively.
The teal curves diverging from $H_2$ and $H_{12}$ start from $A=24.68$ near $A_2$, and $A=147$ near $A_{12}$\,.

In figure \ref{fig:6.4}, with delay $T=0.9$, amplitudes $A_n$ of the stable orange reference periodic solutions are
$A_{27}=111.25102887868052589...$ and $A_{51}=210.13193020360773942...$;
amplitudes of the unstable teal periodic solutions are 
$A_{28}=115.35833191723956861...$ and $A_{52}=214.24522922435665376...\ $.
Initial conditions are $A=110.1$ and $A=209.1$  for the orange curves converging to $H_{27}$ and $H_{51}$\,, respectively.
The teal curves diverging from $H_{28}$ and $H_{52}$ start from $A=115.2$ near $A_{28}$\,, and $A=214$ near $A_{52}$\,.

Lower plateaus in the logarithmic plots indicate residual relative numerical errors of our numerical simulations. 
The local relative error tolerance of \texttt{dde23} is $10^{-3}$. 
All simulations support the theoretically predicted Floquet exponent $\tilde{\eta} = (-1)^n T/3$ of \eqref{eq:6.1}, which corresponds to the slopes $ \sim \pm 0.1$ for $T=0.3$, in figure \ref{fig:6.3} (c), (d), and slopes $\sim \pm 0.3$ for $T=0.9$ in figure \ref{fig:6.4} (c), (d). 
Slopes were determined by least square fits (gray lines).
Note how the slopes only depend on the even/odd parity, but not on the value, of $n$, asymptotically for large $n$.
Given that our original expansion \eqref{eq:5.2} was limited to ``sufficiently small'' $T$ and ``large enough'' $n$, we are rather surprised at such quantitative agreement far from those limits.

%%%%%%%%%%%%%%%%%%%%%%%%%%%%%%

\begin{figure}[t] 
\centering
	\includegraphics[width=\textwidth]{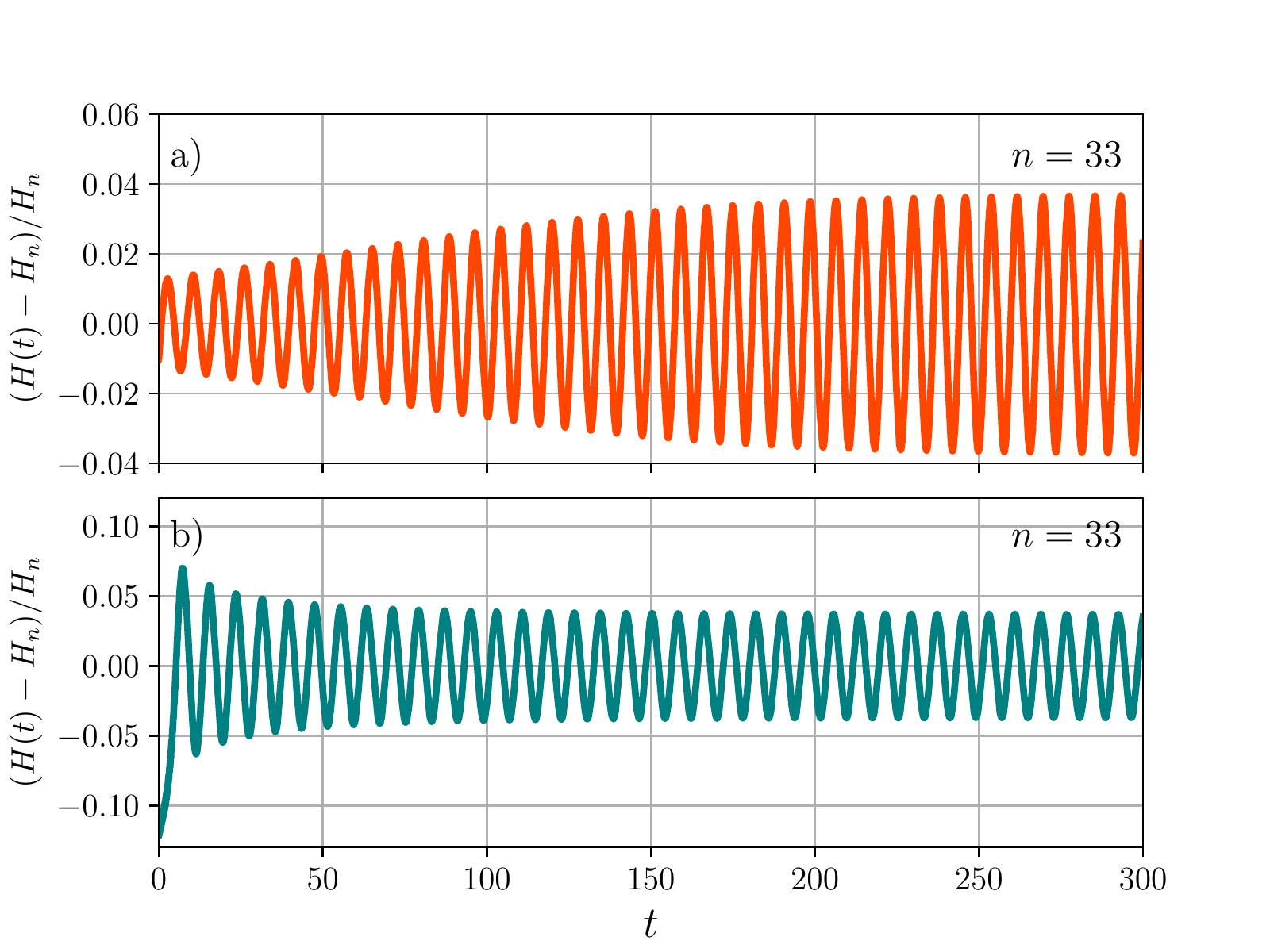}
	\caption{Transients to asymptotically periodic oscillations of the energy H(t) for delays $T=T_\textsf{crit}+0.1$ slightly above the critical threshold $T_\textsf{crit}=\sqrt{3/2}\,\pi$.
	Horizontal axes $0\leq t \leq 300$ in both plots.
	The oscillations indicate bifurcation of an invariant 2-torus $\mathbb{T}^2$ from the periodic solution $x_{33}(t)$, for delay $T$ near $T_\textsf{crit}$\,.
	Top (orange): Initial amplitude $A=31.1$ slightly above $A_{33}=31.021...$ indicates oscillatory stability loss of $x_{33}(t)$. 
	Bottom (teal): Initial amplitude $A=31.9$ slightly below $A_{34}= 31.944...$ leads to a transition to oscillatory $H(t)$, due to persisting real instability of $x_{34}(t)$.}
	\label{fig:6.5} 
\end{figure}

%%%%%%%%%%%%%%%%%%%%%%%%%%%%%%

%%%%%%%%%%%%%%%%%%%%%%%%%%%%%%

\begin{figure}[t] 
\centering
	\includegraphics[width=1.05\textwidth]{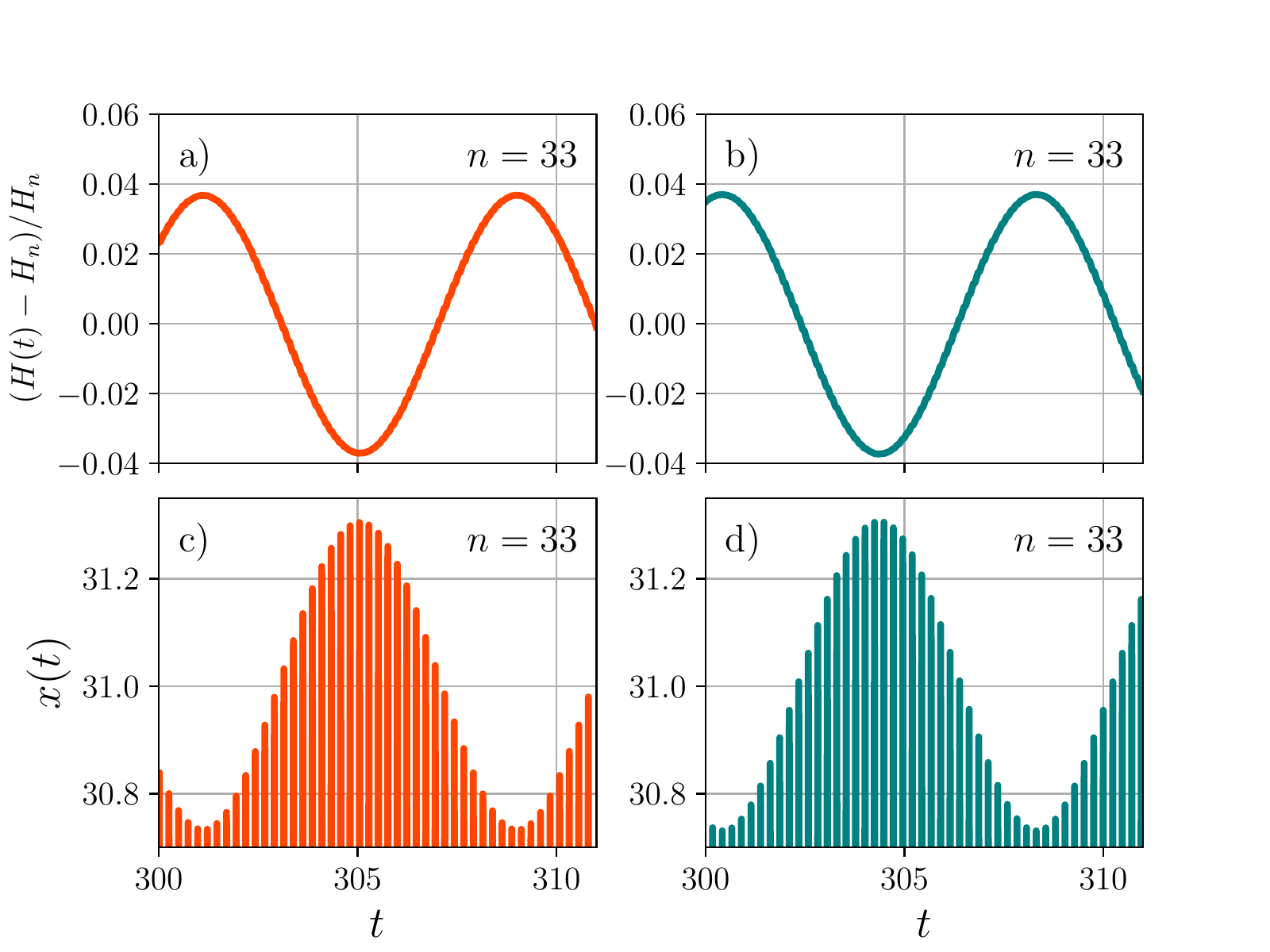}
	\caption{
	For the same initial amplitudes $A$ and color coding as in figure \ref{fig:6.5}, we focus on the same post-transient time interval $300\leq t \leq 311$ (horizontal axis), in all four plots. 
	Top row: Close-up of the sinusoidal oscillations of $H(t)$.
	Bottom row: possibly still $1:n=1:33$ resonant fluctuation of $x(t)$ in the peak region $30.7\leq x \leq 31.35$. 
	Left column (orange): Initial amplitude $A=31.1$ slightly above $A_{33}=31.021...$ indicates approach from $x_{33}(t)$. 
	Right column (teal):  Initial amplitude $A=31.9$ slightly below $A_{34}= 31.944...$ indicates approach from $x_{34}(t)$. 
	Both columns essentially agree, up to time-shift.
	The identical slowly periodic  fluctuations of $H(t)$, and the rapid oscillations of $x(t)$ with slowly fluctuating amplitude,  indicate convergence to the same invariant 2-torus $\mathbb{T}^2$.
}
	\label{fig:6.6} 
\end{figure}

%%%%%%%%%%%%%%%%%%%%%%%%%%%%%%

Our final figures are testing for the conjectural torus bifurcation at the critical boundary 
\begin{equation}
\label{eq:6.9}
T=T_\textsf{crit}=\sqrt{\tfrac{3}{2}}\,\pi = 3.8476494904855922866...	
\end{equation}
of assumption \eqref{eq:1.19} in theorem \ref{thm:1.1}, for odd $n$ and $b=1$.
We check for oscillatory $H(t)$ at $T=T_\textsf{crit}+0.1\,$.

In figure \ref{fig:6.5} we plot the relative deviations $(H(t)-H_n)/H_n$ for $n=33$.
On the top (orange), we start at an initial amplitude of $A=31.1$, slightly above the amplitude $A_{33}=31.021414799836585...$ of the reference rapid periodic solution $x_{33}(t)$.
We clearly observe a loss of stability of the solution $x_{33}(t)$, which we asserted to be stable for $T<T_\textsf{crit}$ and large enough $n$.
In fact $H(t)-H_{33}$ increases to an asymptotically periodic, sinusoidal oscillation. 
This suggests a supercritical Neimark-Sacker bifurcation  \cite{iooss,neisa} of $x_{33}(t)$ to a stable 2-torus $\mathbb{T}^2$, near $T=T_\textsf{crit}$\,. Our remark following the proof of proposition \ref{Prop 5.2} predicts that the bifurcation occurs near $1:n=1:33$ resonance.
This may result in an asymptotic oscillation of $x(t)$ itself which is still $1:33$ subharmonic, or possibly quasiperiodic.

On the bottom (teal), we start at an initial amplitude $A=31.9$ slightly below the reference amplitude $A_{34}=31.91443613945749...\,$. 
Real instability of $x_{34}(t)$ persists to dominate, and we observe an asymptotic decay to the same sinusoidal periodic oscillation of $H(t)-H_{33}$\,, as in the previous case. 
This further attests to the presence of a stable invariant 2-torus $\mathbb{T}^2$, which causes the asymptotically $1:33$ resonant subharmonic, or possibly quasiperiodic, oscillation of $x(t)$, further examined in figure \ref{fig:6.6}.

For post-transient times $300 \leq t \leq 311$, we examine a close-up on the sinusoidal oscillations of $-0.04\leq (H(t)-H_{33}))/H_{33} \leq +0.04$ and the subharmonic, or possibly quasiperiodic, fluctuation of $x(t)$ in the peak region $30.7\leq x \leq 31.35$. 
See the top and bottom rows (a), (b) and (c), (d) of figure \ref{fig:6.6}, respectively. 
The orange graphs, in the left column, and the teal graphs, in the right column, refer to the same initial conditions as in figure \ref{fig:6.5}.
The top row clearly indicates convergence of both solutions to the same invariant 2-torus $\mathbb{T}^2$, with identical sinusoidal oscillations of $H(t)-H_{33}$ up to a phase shift. 
In particular the asymptotic periods $\sim 2T$ coincide, right and left.
The sinusoidal character of $H(t)-H_{33}$ indicates that our delay parameter $T$ is close to the actual bifurcation point, where the invariant 2-torus $\mathbb{T}^2$ bifurcates from the destabilizing rapidly periodic reference solution $x_{33}(t)$.

The bottom row of figure \ref{fig:6.6} shows a slow sinusoidal fluctuation, over slow periods $\sim 2T$, of the amplitudes of the rapid oscillations of $x(t)$ of minimal ``periods'' near $2T/n= 0.23925...\ $. 
This indicates the $1:n=1:33$ subharmonic, or possibly quasiperiodic, flow on the invariant 2-torus $\mathbb{T}^2$, and agrees well with our remark following the proof of proposition \ref{Prop 5.2}.

%%%%%%%%%%%%%%%%%%%%%%%%%%%%%%%%%%%%%%%%%%%%%%%%%%%%%%%

\end{document}